\documentclass[11pt,reqno]{amsart}
\usepackage{amscd,amsmath,amsxtra,amsthm,amssymb,stmaryrd,xr,mathrsfs,mathtools,enumerate,commath, comment, enumitem}
\usepackage{stmaryrd}
\usepackage{xcolor}
\usepackage{commath}
\usepackage{comment}
\usepackage{tikz-cd}
\usepackage[top=25mm,bottom=25mm,left=27mm,right=27mm]{geometry}
\usepackage{longtable} 
\usepackage{pdflscape} 
\usepackage{booktabs}
\usepackage{hyperref}
\definecolor{vegasgold}{rgb}{0.77, 0.7, 0.35}
\definecolor{darkgoldenrod}{rgb}{0.72, 0.53, 0.04}
\definecolor{gold(metallic)}{rgb}{0.83, 0.69, 0.22}
\hypersetup{
 colorlinks=true,
 linkcolor=darkgoldenrod,
 filecolor=brown,      
 urlcolor=gold(metallic),
 citecolor=darkgoldenrod,
 }

\usepackage[all,cmtip]{xy}

\DeclareFontFamily{U}{wncy}{}
\DeclareFontShape{U}{wncy}{m}{n}{<->wncyr10}{}
\DeclareSymbolFont{mcy}{U}{wncy}{m}{n}
\DeclareMathSymbol{\Sh}{\mathord}{mcy}{"58}
\usepackage[T2A,T1]{fontenc}
\usepackage[OT2,T1]{fontenc}

\newtheorem{theorem}{Theorem}[section]
\newtheorem{lemma}[theorem]{Lemma}

\newtheorem*{theorem*}{Theorem}
\newtheorem*{ass*}{Assumption}
\newtheorem{definition}[theorem]{Definition}

\newtheorem{remark}[theorem]{Remark}

\newtheorem{proposition}[theorem]{Proposition}

\newcommand{\cF}{\mathcal{F}}

\newcommand{\cG}{\mathcal{G}}

\newcommand{\Z}{\mathbb{Z}}

\newcommand{\Q}{\mathbb{Q}}
\newcommand{\F}{\mathbb{F}}

\newcommand{\cL}{\mathcal{L}}

\newcommand{\cO}{\mathcal{O}}
\newcommand{\cS}{\mathcal{S}}

\newcommand{\coker}{\mathrm{coker}}

\newcommand{\op}[1]{\operatorname{#1}}

\numberwithin{equation}{section}

\begin{document}
\title[Selmer stability for elliptic curves]{Selmer stability for elliptic curves in Galois $\ell$-extensions}



\author[S.~Pathak]{Siddhi Pathak}
\address[Pathak]{Chennai Mathematical Institute, H1, SIPCOT IT Park, Kelambakkam, Siruseri, Tamil Nadu 603103, India}
\email{siddhi@cmi.ac.in}

\author[A.~Ray]{Anwesh Ray}
\address[Ray]{Chennai Mathematical Institute, H1, SIPCOT IT Park, Kelambakkam, Siruseri, Tamil Nadu 603103, India}
\email{anwesh@cmi.ac.in}



\thanks{Research of the first author was partially supported by an INSPIRE faculty fellowship. \\
\textbf{Data availability statement:} Data sharing not applicable – no new data generated, or the article describes entirely theoretical research. \\
\textbf{Conflict of interest statement:} The authors do not have any conflict of interest to declare.}

\begin{abstract}
We study the behavior of Selmer groups of an elliptic curve \( E/\Q \) in finite Galois extensions with prescribed Galois group. Fix a prime \( \ell \geq 5 \), a finite group \( G \) with $\#G = \ell^n$, and an elliptic curve \( E/\Q \) with \( \operatorname{Sel}_\ell(E/\Q) = 0 \) and surjective mod-\( \ell \) Galois representation. We show that there exist infinitely many Galois extensions \( F/\Q \) with Galois group \( \operatorname{Gal}(F/\Q) \simeq G \) for which the \( \ell \)-Selmer group \( \operatorname{Sel}_\ell(E/F) \) also vanishes. We obtain an asymptotic lower bound for the number \( \mathcal{M}(G, E; X) \) of such fields $F$ with absolute discriminant \( |\Delta_F|\leq X \), proving that there is an explicit constant $\delta>0$ such that
\[
\mathcal{M}(G, E; X) \gg X^{\frac{1}{\ell^{n-1}(\ell - 1)}} (\log X)^{\delta - 1}.
\]
 The asymptotic for $\mathcal{M}(G, E; X)$ matches the conjectural count for all $G$-extensions $F/\Q$ for which $|\Delta_F|\leq X$, up to a power of $\log X$. This demonstrates that Selmer stability is not a rare phenomenon.
\end{abstract}

\subjclass[2010]{}

\maketitle

\section{Introduction}
\label{section:intro}

\subsection{Background and motivation}
\par Let \( E \) be an elliptic curve defined over a number field \( K \). The group \( E(K) \) of \( K \)-rational points, known as the Mordell–Weil group, is finitely generated, and its rank is a fundamental arithmetic invariant. One fruitful approach to studying the structure of \( E(K) \) is through Galois cohomology. For a fixed integer \( n \geq 1 \), the \( n \)-torsion subgroup \( E[n] \subset E(\overline{K}) \) carries a natural action of the absolute Galois group \( \operatorname{G}_K := \operatorname{Gal}(\overline{K}/K) \). The Selmer group \( \operatorname{Sel}_n(E/K) \subset H^1(K, E[n]) \), defined by imposing local conditions at all primes of \( K \), sits in a natural short exact sequence,\[0\rightarrow E(K)/nE(K)\rightarrow \op{Sel}_n(E/K)\rightarrow \Sh(E/K)[n]\rightarrow 0\]
encoding information about both the Mordell–Weil group and the Tate–Shafarevich group \( \Sh(E/K) \).

\par A central question in arithmetic geometry is to understand how the Mordell–Weil rank of a fixed elliptic curve varies in field extensions. For instance, given an elliptic curve \( E/\mathbb{Q} \), it is natural to ask how often its rank remains the same in a finite extension \( L/\mathbb{Q} \). More specifically, one studies how frequently \( \operatorname{rank} E(L) = \operatorname{rank} E(\mathbb{Q}) \) as \( L \) ranges over number fields of fixed degree, ordered by discriminant. 

\par In particular, these diophantine stability questions are studied for cyclic extensions of $\Q$ with prime degree. Mazur and Rubin \cite{MazurRubinQtwists} studied the distribution of 2-Selmer ranks in families of quadratic twists of a fixed elliptic curve. Their techniques yield results on the stability of rank across large sets of quadratic twists, and under certain conditions, also demonstrate that arbitrarily large ranks can arise after base change by quadratic extensions. For any prime \( \ell \), Klagsbrun--Mazur--Rubin \cite{KMR}, Mazur--Rubin \cite{mazur2018diophantine} have studied the behavior of \( \ell \)-Selmer groups in \( \mathbb{Z}/\ell\mathbb{Z} \)-extensions, establishing patterns of stability and growth. Smith has recently obtained similar results using arithmetic statistics and the geometry of numbers \cite{smith2022distribution1, smith2022distribution2}. 

The broader phenomenon of rank growth and stability in families of number field extensions has attracted significant recent attention. For example, see the work of Lemke-Oliver and Thorne \cite{Oliverthorne}, Shnidman and Weiss \cite{ShnidmanWeiss}, Beneish--Kundu--Ray \cite{BKR}, Berg--Ryan--Young \cite{Berg}, Keliher \cite{keliher}, Pathak--Ray \cite{pathakray} and Keliher--Park \cite{keliherpark}. 

Understanding Galois extensions of a fixed number field with a prescribed Galois group is a fundamental problem in number theory. Conjectures predicting the asymptotic number of such extensions, ordered by their absolute discriminant, were proposed by Malle \cite{Malle1, Malle2}. These conjectures may be viewed as statistical refinements of the classical inverse Galois problem. Substantial progress towards this has been made in special cases. For example, the conjecture is known for abelian groups, due to the work of Mäki \cite{maki1985density} and Wright \cite{wright1989distribution}, and more recently for certain nilpotent groups by Koymans and Pagano \cite{koymans2023malle}. Although the inverse Galois problem for finite solvable groups over number fields was resolved by Shafarevich \cite{Shaferevich}, Malle’s conjecture remains open for general solvable groups.

\subsection{Main results}
In this article, we apply Galois-theoretic methods to study the behavior of the Mordell–Weil rank and Selmer group of an elliptic curve \( E/\Q \) in certain families of Galois extensions \( L/\Q \) with a prescribed Galois group. Fix a prime \( \ell \geq 5 \), a finite \( \ell \)-group \( G \), and an elliptic curve \( E/\Q \). Assume that \( \operatorname{Sel}_\ell(E/\Q) = 0 \), or equivalently, $E(\Q)[\ell]=0$, $\op{rank}E(\Q)=0$ and $\Sh(E/\Q)[\ell]=0$. Further, suppose that the Galois representation
\[
\rho_{E,\ell} : \operatorname{Gal}(\overline{\Q}/\Q) \longrightarrow \operatorname{GL}_2(\Z/\ell\Z)
\]
associated to \( E[\ell] \) is surjective. Under these assumptions, we show that there exist infinitely many Galois extensions \( L/\Q \) with \( \operatorname{Gal}(L/\Q) \simeq G \) such that \( \operatorname{Sel}_\ell(E/L) = 0 \). Moreover, we obtain an asymptotic lower bound for the number of such extensions ordered by their absolute discriminant.

For each real number \( X > 0 \), let \( \mathcal{N}(G; X) \) denote the number of Galois extensions \( L/\Q \) with Galois group \( G \) and discriminant satisfying \( |\Delta_L| \leq X \). The Hermite–Minkowski theorem ensures that \( \mathcal{N}(G; X) \) is finite for every \( X \), and hence is well-defined.

Let $\cG$ be a finite group and \( \pi_\cG : \cG \hookrightarrow S_{\#\cG} \) denote the regular representation of \( \cG \). For \( g \in \cG \), define the index of \( g \), denoted \( \operatorname{ind}(g) \), as
\[
\operatorname{ind}(g) := \#\cG - \#\text{ orbits of } \pi_\cG(g).
\]
We set $a(\cG) := \left( \min_{g \neq 1} \operatorname{ind}(g) \right)^{-1}$ and note that when $\cG$ is an $\ell$-group with $\ell^n$ elements, $a(\cG)= \left( \ell^{n-1}(\ell - 1) \right)^{-1}$. The weak form of Malle’s conjecture (cf. \cite[p.316]{Malle1}) predicts that for any \( \epsilon > 0 \), there exist constants \( c_1(\cG), c_2(\cG; \epsilon) > 0 \) such that
\[
c_1(\cG)\, X^{a(\cG)} \leq \mathcal{N}(\cG; X) \leq c_2(\cG; \epsilon)\, X^{a(\cG) + \epsilon}
\]
for all sufficiently large \( X \). This has been established for all finite nilpotent groups by Kl\"uners and Malle \cite{klunersmalle}. For finite nilpotent groups satisfying additional hypotheses, Koymans and Pagano \cite{koymans2023malle} have proven a stronger version of Malle's conjecture, obtaining the exact asymptotic.

On the other hand, let \( \mathcal{M}(\cG, E; X) \) denote the number of Galois extensions \( F/\Q \) with \( \operatorname{Gal}(F/\Q) \simeq \cG \), \( |\Delta_F| \leq X \), and \( \operatorname{Sel}_\ell(E/F) = 0 \). Note that for such fields $F/\Q$, $\op{rank}E(F)=0$ and $\Sh(E/F)[\ell]=0$. Thus, the Mordell--Weil rank and the $\ell$-part of the Tate-Sharevich group of $E$ remain stable in $F/\Q$. For a group $G$ with $\# G=\ell^n$, our main result establishes asymptotic lower bounds for \( \mathcal{M}(G, E; X) \) as \( x \to \infty \), showing that it grows comparably to \( \mathcal{N}(G; X) \). This demonstrates that the vanishing of the \( \ell \)-Selmer group over such \( G \)-extensions is not a rare phenomenon.

\begin{theorem}\label{main thm of section 4}
Let \( \ell \geq 5 \) be a prime number, and let \( E/\Q \) be an elliptic curve with \( \operatorname{Sel}_\ell(E/\Q) = 0 \). Assume further that the representation \( \rho_{E,\ell} \) is surjective. Then
\[
\mathcal{M}(G, E; X) \gg X^{\frac{1}{\ell^{n-1}(\ell - 1)}} (\log X)^{\delta - 1} = X^{a(G)} (\log X)^{\delta - 1},
\]
where \( \delta := \frac{\ell^2 - \ell - 1}{\ell^{n-1}(\ell^2 - 1)} \).
\end{theorem}

It follows that the asymptotic growth of \( \mathcal{M}(G, E; X) \) matches that of \( \mathcal{N}(G; X) \) up to a power of \( \log X \). In particular,
\[
\lim_{X \to \infty} \frac{\log \mathcal{M}(G, E; X)}{\log \mathcal{N}(G; X)} = 1.
\]
For $\ell=5$, it is proven by Bhargava and Shankar \cite{bhargavashankar5} that there is a positive density of elliptic curves (ordered by height) $E_{/\Q}$ for which $\op{Sel}_\ell(E/\Q)=0$. On the other hand, a result of Duke \cite{duke} shows that $\rho_{E,\ell}$ is surjective for almost all elliptic curves $E_{/\Q}$. Thus for $\ell=5$, the conditions of Theorem \ref{main thm of section 4} hold for a positive density of elliptic curves over $\mathbb{Q}$.
\par It is conceivable that the results of this article can be generalized to elliptic curves defined over arbitrary number fields. However, for ease of exposition, we restrict ourselves in this paper to the case of elliptic curves defined over $\Q$.

\subsection{Methodology}
Let \( L/K \) be an \( \ell \)-cyclic extension of number fields, and let \( E \) be an elliptic curve defined over \( \Q \), satisfying the hypotheses of Theorem \ref{main thm of section 4}. Suppose in addition that the \( \ell \)-Selmer group of \( E \) over \( K \) vanishes. Analyzing local to global control arguments for Selmer groups, we show that under certain ramification and splitting conditions for the extension $L/K$, the implication
\[
\operatorname{Sel}_\ell(E/K) = 0  \Rightarrow  \operatorname{Sel}_\ell(E/L) = 0
\]
holds. The argument then follows via induction on the length of $G$. The inverse Galois problem for finite $\ell$-groups is known from the work of Reichardt \cite{Reichardt} and Scholz \cite{Scholz}. The Galois cohomological strategy allows us to inductively construct towers of number fields 
\[\Q=L_0\subset L_1\subset \dots \subset L_{n-1}\subset L_n=F,\] where $\op{Gal}(F/\Q)\simeq G$ and $L_i/L_{i-1}$ is an $\ell$-cyclic extension. Furthermore, extending the methods of Kl\"uners and Malle, we are able to construct many such extensions, for which at each stage, the implication 
\[\op{Sel}_\ell(E/L_{i-1})=0\Rightarrow \op{Sel}_\ell(E/L_{i})=0\]
is satisfied. A synthesis of methods from Galois cohomology, arithmetic statistics and analytic number theory give us an asymptotic lower bound, in terms of $X$, for the number of such extensions $F/\Q$ with $|\Delta_F|\leq X$.

\subsection{Organization}
Including the introduction, the article is organized into four sections. Section \ref{s 2} revisits the classical strategy of Reichardt \cite{Reichardt} and Scholz \cite{Scholz} for resolving the inverse Galois problem for finite \( \ell \)-groups over \( \Q \). Our emphasis is on making explicit which primes split or ramify in such extensions. This information is essential for analyzing the behavior of Selmer groups in Galois towers. In Section \ref{s 3}, we establish sufficient conditions on an \( \ell \)-cyclic extension \( L/K \) under which the vanishing of \( \operatorname{Sel}_\ell(E/K) \) implies the vanishing of \( \operatorname{Sel}_\ell(E/L) \). These build upon previous constructions in \cite[Section 2]{pathakray}. We conclude this section by proving the existence of infinitely many \( G \)-extensions \( L/\Q \) such that \( \operatorname{Sel}_\ell(E/L) = 0 \). Finally, Section \ref{s 4} is devoted to the proof of Theorem \ref{main thm of section 4}. We recall key results of Kl\"uners and Malle concerning the enumeration of \( G \)-extensions of \( \Q \) ordered by absolute discriminant. The problem is reduced to counting global Galois cohomology classes subject to prescribed local conditions. To estimate the size of these cohomology groups, we invoke a formula of Wiles which yields precise asymptotics, allowing us to deduce the desired lower bounds.

\section{ The inverse Galois problem for finite $\ell$-groups}\label{s 2}
\medskip 

\par Shafarevich proved that given any number field $K$ and a finite solvable group $G$, there exists a Galois extension $L/K$ with $\op{Gal}(L/K)\simeq G$ (see \cite{Shafarevich} and \cite[Chapter IX, Section 6]{NSW}). Let $\ell$ be an odd prime number. Given a natural number $k$, let $\mu_{k}$ denote the set of $k$-th roots of unity. Fix throughout this article a finite group $G$ with $\# G=\ell^n$. In this special case, the inverse Galois problem over $\Q$ for the group $G$ is due to Reichardt \cite{Reichardt} and Scholz \cite{Scholz}. An exposition of their method can be found in \cite{massey2006inverse}. In this section, we review this construction with a view towards its application in establishing diophantine stability results, for which we need an understanding of the ramification of primes. More precisely, we shall see that these sets of primes are prescribed by \emph{Chebotarev conditions}. \\

\subsection{The embedding problem}
We show that there is a Galois extension $L/\Q$ with $\op{Gal}(L/\Q)\simeq G$. We note that $G$ has a nontrivial center. Thus by induction on $n=\op{length}(G)$, there is a filtration of $G$ by normal subgroups $G_i$:
\begin{equation}\label{filtration on G}\{1\}=G_n\subset G_{n-1}\subset \dots \subset G_{1}\subset G_0=G\end{equation} such that for all $i$,
\begin{itemize}
    \item $G_{i-1}/G_{i}\simeq \Z/\ell\Z$,
    \item $1\to \Z/\ell\Z\to G/G_i\to G/G_{i-1}\to 1$ is a central extension of $G/G_{i-1}$. 
\end{itemize}
 We construct Galois extensions
\[K=L_0\subset L_1\subset L_2\subset \dots L_{n-1}\subset L_n=L\] such that $\op{Gal}(L_i/K)\simeq G/G_i$. We assume that $n\geq 1$ and reduce the solution to that of an \emph{embedding problem}.\\ 

Let $G$ be a finite $\ell$-group and let $\widetilde{G}$ be a central extension of $G$
\begin{equation}\label{extension of G}1\to \Z/\ell\Z\xrightarrow{\iota} \widetilde{G} \xrightarrow{\pi} G\rightarrow 1.\end{equation} Let $\op{G}_{\Q}:=\op{Gal}(\overline{\Q}/\Q)$ and assume that there exists a Galois extension $L/\Q$ such that ${\op{Gal}(L/\Q)\simeq G}$. This gives rise to a surjection $\varphi:\op{G}_{\Q}\rightarrow G$ such that $\overline{\Q}^{\op{ker}\varphi}=L$. The embedding problem asks whether $L$ can be embedded into a larger Galois extension $\widetilde{L}/\Q$ along with an isomorphism $\op{Gal}(\widetilde{L}/\Q)\simeq \widetilde{G}$, such that the following diagram commutes:
\[
\begin{array}{ccccccccc}
 & 0 & \to & \op{Gal}(\widetilde{L}/L) & \to & \op{Gal}(\widetilde{L}/\Q) & \to & \op{Gal}(L/\Q) & \to 0 \\
 &   &     & \downarrow &   & \downarrow &   & \downarrow &  \\
 & 0 & \to & \Z/\ell \Z  & \to & \widetilde{G}  & \to & G  & \to 0.
\end{array}
\]
In the above diagram, the downward arrows are isomorphisms. Equivalently, the embedding problem is solvable if $\varphi$ can be lifted to a surjective homomorphism $\widetilde{\varphi}: \op{G}_\Q\twoheadrightarrow \widetilde{G}$. Indeed, setting $\widetilde{L}:=\overline{\Q}^{\op{ker}\widetilde{\varphi}}$ we find that $\op{Gal}(\widetilde{L}/\Q)\simeq \widetilde{G}$. Note that if the extension \eqref{extension of G} is non-split then any lift $\widetilde{\varphi}$ is surjective. In order to inductively construct the extensions $L_i$, it suffices to solve the embedding problem for 
\[1\rightarrow \Z/\ell \Z\rightarrow G/G_{i+1}\rightarrow G/G_i\rightarrow 1\] for each $i$ in the range $0\leq i <n$. \\

\subsection{Cohomological parametrization of extension classes}
\par We recall how to parametrize group extensions by cohomology classes. Let $A$ be a given abelian group and $G$ be an arbitrary finite group, and consider a central extension of $G$ by $A$:
\[1\rightarrow A\rightarrow \widetilde{G}\xrightarrow{\pi} G\rightarrow 1.\]
Then $A$ has an induced $G$-module structure which we describe. Let $\eta: G\rightarrow \widetilde{G}$ be a set theoretic section of the map $\pi: \widetilde{G}\rightarrow G$. Since $\pi$ is surjective, such a map exists. We define the action of $G$ on $A$ as follows: for $a\in A$ and $\sigma\in G$,
\[\sigma\cdot a:=\eta(\sigma) \, a \, \eta(\sigma)^{-1}.\] There are a few things to take note of here. First, since $A$ is a normal subgroup of $G$, the element $\eta(\sigma) \, a \, \eta(\sigma)^{-1}$ is in $A$. Next, since $A$ is abelian, $\sigma\cdot a$ is independent of the choice of set theoretic section $\eta$. From here on, when referring to $A$ as a $G$-module, it will be with respect to this chosen action, and $H^i(G, A)$ will be the associated cohomology groups. Two extensions $\widetilde{G}_{1}$ and $\widetilde{G}_{2}$ are said to be equivalent if there is an isomorphism $f:\widetilde{G}_1\xrightarrow{\sim} \widetilde{G}_2$ such that the following diagram commutes:
$$
\begin{array}{ccccc}
A & \longrightarrow & \widetilde{G}_{1} & \longrightarrow & G  \\
\| & & \downarrow{f} & & \| \\
A & \longrightarrow & \widetilde{G}_{2} & \longrightarrow & G.
\end{array}
$$

\begin{proposition}
    Let $G$ be a finite group and $A$ be an abelian group which is also a $G$-module. There is a bijection between
    \begin{itemize}
        \item equivalence classes of $A$-extensions of $G$:
        \[1\mapsto A\xrightarrow{\iota} \widetilde{G}\xrightarrow{\pi} G\rightarrow 1\]
        such that for any set theoretic section $\eta:G\rightarrow \widetilde{G}$ of $\pi$, we have that $\sigma\cdot a=\eta(\sigma) \, a \, \eta(\sigma)^{-1}$
        \item elements of $H^2(G, A)$. 
    \end{itemize}
    The class associated to $\widetilde{G}$ is denoted by $\theta_{\widetilde{G}}\in H^2(G, A)$. Moreover, the association has the property that the \emph{trivial} class $\widetilde{G}:=G\ltimes A$ corresponds to the trivial element in $H^2(G, A)$.
\end{proposition}
\begin{proof}
    This is a standard result, cf. \cite[Theorem 6.6.3]{weibel}.
\end{proof}

\subsection{The method of Reichardt and Scholz}
\begin{definition}Let $G$ be an $\ell$-group with $\# G = \ell^n$. Let \( L / \mathbb{Q} \) be a Galois extension with \( \op{Gal}( L/\Q) \simeq G \) and choose \( N \geq n \). The extension $L/\Q$ satisfies the \emph{Scholz property} for \( N \), denoted \( (\mathfrak{S}_N) \), if:
\begin{itemize}
    \item every rational prime \( p \) ramified in $L$ satisfies \( p \equiv 1 \pmod{\ell^N} \),
    \item for each prime \( v|p\) of \( L \), we have that \( L_v / \mathbb{Q}_p \) is totally ramified.
\end{itemize}
\end{definition}
Assuming that there exists a Galois extension $L/\Q$ satisfying $(\mathfrak{S}_N)$ with $\op{Gal}(L/\Q)\simeq G$, we embed $L$ into a larger Galois extension $\widetilde{L}$ with $\op{Gal}(\widetilde{L}/\Q)\simeq \widetilde{G}$, also satisfying $(\mathfrak{S}_N)$.
\begin{theorem}\label{inverse-Gal-thm}
Let $\ell$ be an odd prime number, $G$ be a finite group with $\# G=\ell^n$, and fix $N\geq n$. Suppose that there exists a Galois extension $L/\Q$ with $\op{Gal}(L/\Q)\simeq G$ which satisfies \( (\mathfrak{S}_N) \). Let $\{p_1, \dots, p_m\}$ denote the primes that ramify in $L$. Then the embedding problem for \( L \) and \( \widetilde{G} \) is solvable. Moreover, the solution \( \widetilde{L} \) can be chosen to satisfy \( (\mathfrak{S}_N) \) with at most one additional prime outside $\{p_1, \dots, p_m\}$ being ramified in $\widetilde{L}$.
\end{theorem}
The proof of Theorem \ref{inverse-Gal-thm} is important for us to outline since the construction will be extended in subsequent sections in which our diophantine stability results are proven. More specifically, the additional ramified prime in $\widetilde{L}$ will be chosen to belong to an infinite set of primes defined by Chebotarev conditions. These conditions will be shown to be compatible with the Chebotarev conditions on primes that ensure Diophantine stability for a given elliptic curve. 

\subsubsection{Proof of Theorem \ref{inverse-Gal-thm} in the split case}\label{split-case-proof-subsubsection}
\par First, we consider the case when \eqref{extension of G} is split. Let $p_1, \dots, p_m$ be the rational primes which are ramified in $L$ and $q$ be a prime number. It is easy to see that $q$ splits completely in the field $L\left(\mu_{\ell^{N}}, \sqrt[\ell]{p_{1}}, \ldots, \sqrt[\ell]{p_{m}}\right)$ if and only if the following conditions are satisfied:
\begin{enumerate}
    \item[(i)] $q \equiv 1\pmod{\ell^{N}}$,
    \item[(ii)] $q$ splits completely in $L$,
    \item[(iii)] $q\notin \{p_1, \dots, p_m\}$ and $p_i$ is an $\ell$-th power in $\F_q^\times$ for $i=1, \dots, m$. 
\end{enumerate}
By the Chebotarev density theorem, there exists a positive density set of primes $\mathcal{P}_{L}$ that satisfy the above conditions. More specifically, this density is given by 
\begin{equation*}
    \delta(\mathcal{P}_L) : =\frac{1}{\left[ L\left(\mu_{\ell^{N}}, \sqrt[\ell]{p_{1}}, \ldots, \sqrt[\ell]{p_{m}}\right): \mathbb{Q}\right]}.
\end{equation*}

Let $q \in \mathcal{P}_L$ and $M_q \subset \Q(\mu_q)$ be the field such that $\op{Gal}(M_q/\Q)\simeq \Z/\ell \Z$. Set $\widetilde{L}:=L\cdot M_q$. Since $q$ is totally ramified in $M_q$ and split in $L$, it follows that $L \cap M_q = \Q$, therefore, $\op{Gal}(\widetilde{L}/\Q)\simeq G\times \Z/\ell \Z=\widetilde{G}$. It is easy to see that $\widetilde{L}$ satisfies $(\mathfrak{S}_N)$. In greater detail, each of the primes $p_i$ are of the form $p_i \equiv 1 \bmod \ell^N$ as $L$ satisfies ($\mathfrak{S}_N$). The prime $q$ is assumed to satisfy $q \equiv 1 \bmod \ell^N$. Thus, all primes that ramify in $\widetilde{L}$ are $\equiv 1 \bmod \ell^N$. Given a prime $p_i$ for $i = 1, \cdots , m$, let $\widetilde{v}$ be a prime of $\widetilde{L}$ that lies above $p_i$. Let $v$ (resp. $v'$) be a prime of $L$ (resp. $M_q$) such that $\widetilde{v}|v$ (resp. $\widetilde{v}|v'$), as depicted below:
 
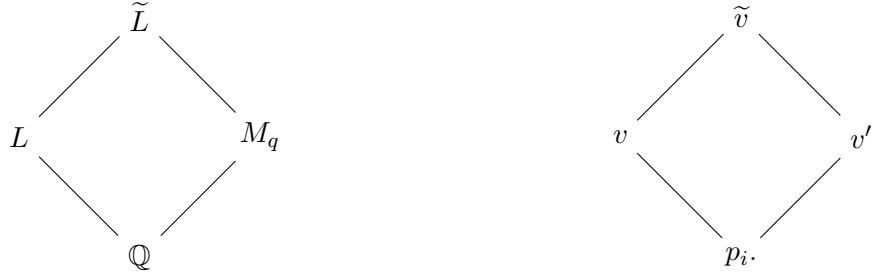
\[ \begin{tikzpicture}[scale=.8]
    \begin{scope}[xshift=0cm]
    \node (Q1) at (0,0) {$\Q$};
    \node (Q2) at (2,2) {$M_q$};
    \node (Q3) at (0,4) {$\widetilde{L}$};
    \node (Q4) at (-2,2) {$L$};

    \draw (Q1)--(Q2) node [pos=0.7, below,inner sep=0.25cm] {};
    \draw (Q1)--(Q4) node [pos=0.7, below,inner sep=0.25cm] {};
    \draw (Q3)--(Q4) node [pos=0.7, above,inner sep=0.25cm] {};
    \draw (Q2)--(Q3) node [pos=0.3, above,inner sep=0.25cm] {};
    \end{scope}

    \begin{scope}[xshift=10cm]
    \node (Q1) at (0,0) {$p_i$.};
    \node (Q2) at (2,2) {$v'$};
    \node (Q3) at (0,4) {$\widetilde{v}$};
    \node (Q4) at (-2,2) {$v$};

    \draw (Q1)--(Q2) node [pos=0.7, below,inner sep=0.25cm] {};
    \draw (Q1)--(Q4) node [pos=0.7, below,inner sep=0.25cm]{};
    \draw (Q3)--(Q4) node [pos=0.7, above,inner sep=0.25cm]{};
    \draw (Q2)--(Q3) node [pos=0.7, above,inner sep=0.25cm]{};
    \end{scope}
    \end{tikzpicture}
\]
Since $L$ satisfies $(\mathfrak{S}_N)$, we have that $f(v \mid p_i)= 1$. By condition (iii) in the choice of $\mathcal{P}_{L}$, $p_i$ splits completely in $M_q$. Therefore, we deduce that $f(\widetilde{v} \mid p_i) = 1$, i.e., $\widetilde{L}_{\widetilde{v}}/\Q_{p_i}$ is totally ramified. Since $q$ splits completely in $L$ and is totally ramified in $M_q$, $f(\widetilde{w} \mid q) = 1$ for any prime $\widetilde{w}$ of $\widetilde{L}$ that lies above $q$. Thus, $\widetilde{L}$ as defined above is a solution the embedding problem and satisfies $(\mathfrak{S}_N)$. The set of primes ramified in $\widetilde{L}$ is precisely $\{p_1, \cdots , p_m$, $q\}$ and therefore, $\widetilde{L}$ is ramified at exactly one additional prime.\\

\subsubsection{Proof of Theorem \ref{inverse-Gal-thm} in the non-split case}\label{non-split-case-proof-subsubsection}
\par Next we focus our attention on the case when \eqref{extension of G} is non-split, or equivalently, $\theta_{\widetilde{G}}\in H^2(G, \Z/\ell \Z)$ is nontrivial. In this case, the extension $\widetilde{L}/L/\Q$ is constructed in three steps.
\begin{description}
    \item[Step 1]  We construct an extension \( \widetilde{L} \) that solves the embedding problem for the desired group extension \( \widetilde{G} \).  
\item[Step 2] We modify \( \widetilde{L} \) so that the set of primes that ramify in $\widetilde{L}$ matches the set of ramified primes of \( L \).  
\item[Step 3] We adjust \( \widetilde{L} \) further to satisfy \( (\mathfrak{S}_N) \), allowing at most one additional ramified prime $q$.
\end{description}
We note that step 1 corresponds to Proposition~\ref{global embedding problem}, step 2 corresponds to Proposition~\ref{step 2 propn}, and step 3 corresponds to Proposition~\ref{constructing tilde L satisfying SN}. Recall that $L$ induces a surjective homomorphism $\varphi: \op{G}_{\Q}\twoheadrightarrow G$ and that there is a bijection between the central extensions $\widetilde{G}$ of $G$ by $\mathbb{Z}/\ell \Z$ and the elements of $H^{2}\left(G, \Z/\ell \Z\right)$. Let $\theta=\theta_{\widetilde{G}} \in H^{2}\left(G, \Z/\ell\Z\right)$ be the element corresponding to the extension $\widetilde{G}$. The map $\varphi$ induces a homomorphism

\begin{equation}\label{defn of phi star}
\varphi^{*}: H^{2}\left(G, \Z/\ell\Z\right) \longrightarrow H^{2}\left(\op{G}_{\Q}, \Z/\ell\Z\right).
\end{equation}

\begin{definition}
    Let \( \cG \), $\cG_1$ and $\cG_2$ be groups, and let \( f_1: \cG_1 \to \cG \) and \( f_2: \cG_2 \to \cG \) be surjective homomorphisms. The fibre product \( \cG_1 \times_\cG \cG_2 \) is the group  
\[
\cG_1 \times_\cG \cG_2 = \{ (g_1,g_2) \in \cG_1 \times \cG_2 \mid f_1(g_1) = f_2(g_2) \}
\]  
with projections \( \pi_1: \cG_1 \times_\cG \cG_2 \to \cG_1 \) and \( \pi_2: \cG_1 \times_\cG \cG_2 \to \cG_2 \). It satisfies the universal property: for any group \( H\) with homomorphisms \( q_1: H \to \cG_1 \) and \( q_2: H \to \cG_2 \) such that \( f_1 \circ q_1 = f_2\circ q_2 \), there exists a unique homomorphism \( q: H \to \cG_1 \times_\cG \cG_2 \) making the diagram below commute:
\[\begin{tikzcd}
{} & H
\arrow[bend right,swap]{ddl}{q_1}
\arrow[bend left]{ddr}{q_2}
\arrow[dashed]{d}[description]{q} & & \\
& \cG_1\times_\cG \cG_2 \arrow{dr}{\pi_2} \arrow{dl}[swap]{\pi_1} \\
\cG_1 \arrow[swap]{dr}{f_1} & & 
\cG_2 \arrow{dl}{f_2} \\
& \cG.
\end{tikzcd}
\]
\end{definition}

\begin{proposition}\label{global lift existence criterion}
    The embedding problem for \eqref{extension of G} is solvable if and only if \( \varphi^*(\theta) = 0 \), where $\varphi^*$ is given by \eqref{defn of phi star}.
\end{proposition}  
\begin{proof}
This is a well known result, we give a sketch of the argument. Setting \(\mathfrak{G}:= \widetilde{G} \times_G \op{G}_{\Q},\) consider the fiber product diagram:
\[
\begin{array}{cccccccc}
1 & \to & \Z/\ell\Z & \to & \mathfrak{G} & \rightarrow & \op{G}_{\Q} & \to 1 \\
& & \| & & \downarrow & & \downarrow{\varphi} & \\
1 & \to & \Z/\ell \Z & \to & \widetilde{G} & \rightarrow & G & \to 1.
\end{array}
\]
Denote by $\pi_1:\mathfrak{G}\rightarrow \widetilde{G}$ and $\pi_2:\mathfrak{G}\rightarrow \op{G}_{\Q}$ the projection maps to the first and second factors. It is easy to see that the top row is a central extension corresponding to \( \varphi^*(\theta) \). This sequence splits if and only if \( \varphi^*(\theta) = 0 \). If the extension splits, there exists a section \( j: \op{G}_{\Q} \hookrightarrow \mathfrak{G}\) to the projection map $\pi_2: \mathfrak{G}\twoheadrightarrow \op{G}_{\Q}$. Then \( \widetilde{\varphi}:= \pi_1 \circ j\) is a lift of $\varphi$. Since it is assumed that \eqref{extension of G} is nonsplit, $\widetilde{\varphi}$ is surjective. The proof of the converse follows along similar lines, and is omitted.
\end{proof}
In what follows, for ease of notation, set $H^i(F, \cdot):=H^i(\op{G}_F, \cdot)$, where $F$ is a field of characteristic $0$. Given a prime number $p$, denote by
\[\op{res}^i_p: H^i(\Q, \Z/\ell\Z)\longrightarrow H^i(\Q_p, \Z/\ell\Z)\] the natural restriction map. This induces a map 
\[\alpha: H^2(\Q, \Z/\ell \Z)\xrightarrow{\bigoplus_p \op{res}_p^2} \bigoplus_{p} H^2(\Q_p, \Z/\ell \Z).\]

\begin{proposition}\label{ H2 injection }
    The map $\alpha$ is an injection.
\end{proposition}
\begin{proof}
    Let \( F \) be a field of characteristic zero that contains the \(\ell\)-th roots of unity, and let \(\op{Br}(F)\) denote its Brauer group. By the standard identification \(\op{Br}(F) \cong H^2(F, \overline{F}^\times)\). Letting \(\op{Br}(F)[\ell]\) be the \(\ell\)-torsion in the Brauer group, there is a natural isomorphism:
\[
H^2(F, \Z/\ell\Z) \cong \op{Br}(F)[\ell].
\]  

\noindent Now, let \( K = \Q(\mu_\ell) \) be the cyclotomic field generated by the \(\ell\)-th roots of unity, and let \( M_K \) denote the set of places of \( K \). By the Brauer–Hasse–Noether theorem, the natural localization map  
\[
\alpha_K: H^2(K, \Z/\ell\Z) \longrightarrow \bigoplus_{v \in M_K} H^2(K_v, \Z/\ell\Z)
\]  
is injective, meaning that an element of \( H^2(K, \Z/\ell\Z) \) is determined by its local invariants.  

Since \( K/\Q \) is a Galois extension with degree prime to \(\ell\), the inflation–restriction sequence for group cohomology shows that the restriction map  
\[
H^2(\Q, \Z/\ell\Z) \longrightarrow H^2(K, \Z/\ell\Z)
\]  
is also injective.  

 These maps fit naturally into the following commutative diagram:  
\[
\begin{array}{ccc}
H^2(\Q, \Z/\ell\Z) & \xrightarrow{\operatorname{res}} & H^2(K, \Z/\ell\Z) \\
\alpha \Big\downarrow & & \Big\downarrow \alpha_K \\
\bigoplus\limits_{v \in M_\Q} H^2(\Q_p, \Z/\ell\Z) & \xrightarrow{\operatorname{res}} & \bigoplus\limits_{v \in M_K} H^2(K_v, \Z/\ell\Z).
\end{array}
\]
We deduce from injectivity of the horizontal restriction maps and the injectivity of $\alpha_K$ that $\alpha$ is injective.
\end{proof}

Given a prime $p$, let \( G_p := \operatorname{Gal}(L_v/\mathbb{Q}_p) \subset G \), where \( v \) is a prime of \( L \) lying over \( p \). A different choice of $v$ gives rise to the same subgroup of $G$ up to conjugation by an element of $G$. Set $\theta_p$ to denote the image of $\theta$ with respect to the natural restriction map: 
\[
\op{res}_p: H^2(G, \mathbb{Z}/\ell\mathbb{Z}) \to H^2(G_p, \mathbb{Z}/\ell\mathbb{Z}).
\]
Consider the associated central extension  
\begin{equation}\label{local embedding problem}
0 \to \mathbb{Z}/\ell\mathbb{Z} \to \widetilde{G}_p \to G_p \to 1
\end{equation}
and let \( \varphi_p: \operatorname{G}_{\mathbb{Q}_p} \twoheadrightarrow G_p \) be the natural quotient map. The \emph{local embedding problem} then asks whether \( \varphi_p \) admits a lift to a homomorphism \( \widetilde{\varphi}_p: \operatorname{G}_{\mathbb{Q}_p} \rightarrow \widetilde{G}_p \). Here we do not insist that $\widetilde{\varphi}_p$ is surjective. By arguments similar to those in the proof of Proposition \ref{global lift existence criterion}, we find that the local embedding problem at $p$ has a solution if and only if $\varphi_p^*(\theta_p)=0$.

\begin{proposition}\label{local global propn}  
    With the notation as above, a surjective lift \(\widetilde{\varphi}: \op{G}_{\Q} \twoheadrightarrow \widetilde{G}\) exists if and only if the local embedding problem is solvable for all primes \( p \).  
\end{proposition}  

\begin{proof}  
By Proposition \ref{global lift existence criterion}, a surjective lift \(\widetilde{\varphi}\) exists if and only if \(\varphi^*(\theta) = 0\). The map $\alpha: H^2(\Q, \Z/\ell\Z)\rightarrow \bigoplus_{p\in M_{\Q}} H^2(\Q_p, \Z/\ell\Z)$ maps $\varphi^*(\theta)$ to the tuple $\alpha(\varphi^*(\theta))=\left(\varphi_p^*(\theta_p)\right)_p$. Proposition \ref{ H2 injection } shows that the map \(\alpha\) is injective, so it suffices to verify that \(\varphi_p^*(\theta_p) = 0\) for all primes \( p \). This is equivalent to the existence of a local lift \(\widetilde{\varphi}_p: \op{G}_{\Q_p} \to \widetilde{G}_p\) at each prime \( p \), which establishes the claim.  
\end{proof}

We now show that the local embedding problem is solvable. We then deduce from Proposition \ref{local global propn} that a surjective lift $\widetilde{\varphi}: \op{G}_{\Q}\twoheadrightarrow \widetilde{G}$ of $\varphi: \op{G}_{\Q}\rightarrow G$ does exist. Given a finite group $\cG$, recall that the Frattini subgroup $\Phi(\cG)$ is the intersection of all proper maximal subgroups of $\cG$. If $\cG$ is an $\ell$-group, then $\Phi(\cG)=\cG^p[\cG, \cG]$, and $\cG/\Phi(\cG)\simeq (\Z/\ell\Z)^d$ for some $d\geq 0$. It is well known that if $\cG$ is a finite $\ell$-group, then $\cG$ is cyclic if and only if the Frattini quotient $\cG/\Phi(\cG)\simeq \Z/\ell\Z$.

\begin{proposition}\label{local extension is solvable}
     Suppose \( L/\mathbb{Q} \) is a Galois extension with \(\op{Gal}(L/\Q)\simeq  G \) and canonical surjection \( \varphi : \op{G}_{\Q}\twoheadrightarrow G\) as before. Then, for every prime \( p \), the local embedding problem \eqref{local embedding problem} is solvable. 
\end{proposition}
\begin{proof}
   If the sequence \eqref{local embedding problem} splits, then $\theta_p=0$. Note that the local embedding problem is solvable if and only if $\varphi_p^*(\theta_p)=0$. Thus we assume without loss of generality that \eqref{local embedding problem} is non-split.  \par First suppose that \( p \) is unramified in \( L \). In this case, $G_p$ is cyclic, generated by the Frobenius element. Let $\Q_p^{\op{nr}}$ be the maximal unramified extension of $\Q_p$ and note that $ \op{Gal}(\Q_p^{\op{nr}}/\Q_p)$ is isomorphic to $\widehat{\Z}$. We seek a lift $\widetilde{\varphi}_p: \widehat{\mathbb{Z}} \rightarrow \widetilde{G}_p$ of $\varphi_p$. In order to show that $\widetilde{\varphi}_p$ exists, it suffices to prove that $\widetilde{G}_p$ is cyclic. Indeed, $G_p$ is cyclic, in particular, abelian. Hence $\Z/\ell\Z$ contains $[\widetilde{G}_p, \widetilde{G}_p]$. If $\Z/\ell\Z=[\widetilde{G}_p, \widetilde{G}_p]$, then, $\Z/\ell\Z$ is contained in $\Phi(\widetilde{G}_p)=[\widetilde{G}_p, \widetilde{G}_p]\widetilde{G}_p^\ell$. Therefore, there is a surjection $\widetilde{G}_p / \left(\Z/\ell\Z\right) \twoheadrightarrow \widetilde{G}_p / \Phi(\widetilde{G}_p)$. In particular, we find that $\widetilde{G}_p / \Phi(\widetilde{G}_p)$ must be cyclic. It follows that $\widetilde{G}_p$ must be cyclic as well. This contradicts the assumption that $[\widetilde{G}_p, \widetilde{G}_p]=\Z/\ell\Z$. Thus, $\widetilde{G}_p$ is abelian. Since $G_p$ is cyclic and the sequence \eqref{local embedding problem} is non-split, this forces $\widetilde{G}_p$ to be cyclic. Therefore, $\varphi_p: \widehat{\Z}\rightarrow G_p$ lifts to a map $\widetilde{\varphi}_p:\widehat{\Z}\rightarrow \widetilde{G}_p$.

\par Next, we consider the case in which $p$ is ramified. Note that $p\equiv 1\pmod{\ell^N}$ by the property $(\mathfrak{S}_N)$ and thus $p\neq \ell$ and the ramification of $p$ in $L$ is tame. Choose a prime $v$ of $L$ which lies above $p$. We have that \( (\cO_{L_{v}} / v)^\times = (\mathbb{Z} / p \mathbb{Z})^\times \). Choose a uniformizer $\pi$ of $L_v$ and consider the homomorphism $\lambda: G_p \to (\mathbb{Z} / p \mathbb{Z})^\times$ defined
by $\sigma \mapsto \overline{\left(\frac{\sigma \pi}{\pi}\right)}$. Since $\Z/p\Z$ is the residue field of $L_v$, it is easy to check that $\lambda$ is a homomorphism. The kernel of $\lambda$ is the wild inertia subgroup of $G_p$, which is trivial. Thus, \( G_p \) injects into \( (\mathbb{Z} / p \mathbb{Z})^\times \). It follows therefore that \( \varphi_p \) factors through a map $\op{Gal}(F/\Q_p)\rightarrow \op{G}_p$, where \( F \) is the maximal abelian tamely ramified extension of \( \mathbb{Q}_p \) with exponent dividing \( \ell^N \). More explicitly, \( F =F_1\cdot F_2\), where $F_1$ is the unramified extension of $\Q_p$ of degree $\ell^N$ and $F_2:=\mathbb{Q}_p(p^{1/\ell^N})$. Because $G_p$ is abelian, it follows (from the same argument as in the unramified case) that $\widetilde{G}_p$ is abelian. Note that \( \operatorname{Gal}(F / \mathbb{Q}_p) \simeq  (\mathbb{Z} / \ell^N \mathbb{Z}) \oplus (\mathbb{Z} / \ell^N \mathbb{Z})\). As $G_p$ is a quotient of $\op{Gal}(F/\Q_p)$, we find that $G_p\simeq \Z/\ell^k \Z$, or $G_p\simeq \Z/\ell^{k_1} \Z\oplus \Z/\ell^{k_2} \Z$ for some $k_1, k_2\geq 1$. Since $\widetilde{G}_p$ is abelian and fits into a non-split sequence \eqref{local embedding problem}, we find that $\widetilde{G}_p\simeq \Z/\ell^{k+1}\Z$ or $\widetilde{G}_p\simeq \Z/\ell^{k_1+1}\Z\oplus \Z/\ell^{k_2}\Z$ respectively. In both cases it is clear that $\varphi_p:\op{Gal}(F/\Q_p)\rightarrow G_p$ lifts to a homomorphism $\widetilde{\varphi}_p:\op{Gal}(F/\Q_p)\rightarrow \widetilde{G}_p$. This completes the proof.
\end{proof}

\begin{proposition}\label{global embedding problem}
    There is a surjective lift $\widetilde{\varphi}: \op{G}_\Q\twoheadrightarrow G$ of $\varphi:\op{G}_{\Q}\rightarrow G$.
\end{proposition}
\begin{proof}
    Proposition \ref{local extension is solvable} implies that for each prime $p$, the local embedding problem is solvable, i.e., there is a lift $\widetilde{\varphi}_p$ of $\varphi_p$. Thus it follows from Proposition \ref{local global propn} that a lift $\widetilde{\varphi}:\op{G}_\Q\rightarrow \widetilde{G}$ of $\varphi: \op{G}_\Q\rightarrow G$ exists. Since the extension 
    \[1\rightarrow \Z/\ell\Z\rightarrow \widetilde{G}\rightarrow G\rightarrow 1\] is non-split, $\widetilde{\varphi}$ is surjective.
\end{proof}

This completes Step 1. We now move on to Step 2 which involves modifying $\widetilde{L}$ so that the set of primes that ramify in $\widetilde{L}$ matches the set of ramified primes in $L$.\\

\par The action of \(\op{G}_{\Q}\) on \(\Z/\ell\Z\) is trivial, allowing us to identify \(H^1(\op{G}_{\Q}, \Z/\ell\Z)\) with \(\op{Hom}(\op{G}_{\Q}, \Z/\ell\Z)\). Let \(\widetilde{L}\) and \(\widetilde{L}'\) be extensions of \(L\) that solve the embedding problem, and let \(\widetilde{\varphi}, \widetilde{\psi}: \op{G}_{\Q} \to \widetilde{G}\) be the corresponding lifts of \(\varphi: \op{G}_{\Q} \to G\) respectively. Then for $\sigma\in \op{G}_\Q$, we have that \[\widetilde{\psi}(\sigma)\widetilde{\varphi}(\sigma)^{-1}\in \Z/\ell\Z=\op{ker}(\widetilde{G}\rightarrow G).\] The map $f:=\widetilde{\psi}\widetilde{\varphi}^{-1}:\op{G}_\Q\rightarrow \Z/\ell\Z$ is a homomorphism. In greater detail, 
\[\begin{split}
f(\sigma_1\sigma_2)=&\widetilde{\psi}(\sigma_1\sigma_2)\widetilde{\varphi}(\sigma_1\sigma_2)^{-1} \\ = & \widetilde{\psi}(\sigma_1)\widetilde{\psi}(\sigma_2)\widetilde{\varphi}(\sigma_2)^{-1}\widetilde{\varphi}(\sigma_1)^{-1}\\
=& \widetilde{\psi}(\sigma_1) \widetilde{\varphi}(\sigma_1)^{-1}\widetilde{\psi}(\sigma_2)\widetilde{\varphi}(\sigma_2)^{-1}=f(\sigma_1)f(\sigma_2)\\
\end{split}\]
wherein the second last equality follows since $\widetilde{\psi}(\sigma_2)\widetilde{\varphi}(\sigma_2)^{-1}$ belongs to the center of $G$. It is clear that $f(1)=1$ and $f(\sigma^{-1})=f(\sigma)^{-1}$. Conversely, suppose that $f\in H^1(\op{G}_\Q, \Z/\ell\Z)$ and $\widetilde{L}/L/\Q$ solves the embedding problem, with $\widetilde{\varphi}:\op{G}_\Q\rightarrow \widetilde{G}$ the corresponding homomorphism. Then, $\widetilde{\psi}:=\widetilde{\varphi} f:\op{G}_\Q\rightarrow \widetilde{G}$ is a well defined surjective homomorphism that lifts $\varphi$. This also gives rise to a solution $\widetilde{L}'/L/K$ of the embedding problem.

\begin{definition}\label{defn of twist} With respect to above notation, we refer to $\widetilde{L}'$ as the \emph{twist} of $\widetilde{L}$ by $f$, and denote it by $\widetilde{L}^f$.
\end{definition}
By Proposition \ref{global embedding problem} a solution $\widetilde{L}/L/\Q$ of the embedding problem always exists. Hence there is a bijection:
\[\left\{\widetilde{L}^f/L/\Q : \widetilde{L}^f\text{ solves the embedding problem} \right\}\leftrightarrow H^1(\op{G}_{\Q}, \Z/\ell\Z). \]
\par Let $C$ be a finite dimensional vector space over $\Z/\ell \Z$ equipped with an action of $\op{G}_{\Q}$. Set $C^*:=\op{Hom}\left(C, \mu_\ell\right)$ with the induced Galois module structure. Given a vector space $V$ over $\Z/\ell \Z$, denote by $V^\vee:=\op{Hom}(V, \Z/\ell\Z)$. A prime $p$ is said to be ramified in $C$ if the inertia group $\op{I}_p\subset \op{G}_{\Q_p}$ acts non-trivially on $C$. The field $\Q(C)$ is the Galois extension of $\Q$ \emph{cut out by $C$}. In greater detail, $\Q(C):=\overline{\Q}^{\op{ker}\rho_C}$ where $\rho_C: \op{G}_{\Q}\rightarrow \op{Aut}(C)$ is the representation of $\op{G}_{\Q}$ on $C$. Note that a prime $p$ is ramified in $C$ if and only if it ramifies in $\Q(C)$. Let $\Sigma_C$ be the set of primes $p$ such that either $p\neq \ell$ and $p$ is ramified in $C$, or $p=\ell$. Fix a finite set of primes $S$ containing $\Sigma_C$ and let $\Q_S$ be the maximal algebraic extension of $\Q$ in which all primes $p\notin S$ are unramified. Given a number field $F\subset \Q_S$, we set $H^i(\Q_S/F, \cdot):=H^i(\op{Gal}(\Q_S/F), \cdot)$ where $i=1,2$. Moreover, we set $H^i(\Q_p, \cdot):=H^i(\op{G}_{\Q_p}, \cdot)$. We define:
\[\Sh^i_S(C):=\op{ker}\left( H^i(\Q_S/\Q, C)\longrightarrow \bigoplus_{p\in S} H^i(\Q_p, C)\right).\] Global duality for $\Sh$-groups \cite[Theorem 8.6.7]{NSW} states that there is a natural isomorphism:
\begin{equation}\label{PT duality}\Sh^2_S(C)\simeq \Sh^1_S(C^*)^\vee.\end{equation}
For each prime $p\in S$ let $\cL_p$ be a subspace of $H^1(\Q_p, C)$ and let $\cL_p^\perp$ be the orthogonal complement of $\cL_p$ with respect to the nondegenerate Tate local duality pairing:
\[\langle \cdot, \cdot \rangle_p : H^1(\Q_p, C)\times H^1(\Q_p, C^*)\xrightarrow{\cup} H^1(\Q_p, \mu_\ell)\xrightarrow{\sim}\Z/\ell \Z. \]
In other words, 
\[\cL_p^\perp:=\{v\in H^1(\Q_p, C^*)\mid \langle w, v\rangle_p=0\text{ for all }w\in \cL_p\}.\]
\begin{definition}\label{defn of Selmer and dual Selmer}
    With respect to notation above, the Selmer and dual Selmer groups, denoted $H^1_{\cL}(\Q_S/\Q, C)$ and $H^1_{\cL^\perp}(\Q_S/\Q, C^*)$ are defined as follows:
    \[\begin{split}
        & H^1_{\cL}(\Q_S/\Q, C):= \op{ker}\left\{ H^1(\Q_S/\Q, C)\longrightarrow \bigoplus_{p\in S} \frac{H^1(\Q_p, C)}{\cL_p}\right\},\\
        & H^1_{\cL^\perp}(\Q_S/\Q, C^*):= \op{ker}\left\{ H^1(\Q_S/\Q, C^*)\longrightarrow \bigoplus_{p\in S} \frac{H^1(\Q_p, C^*)}{\cL_p^\perp}\right\}.\\
    \end{split}\]
\end{definition}
A well known formula due to Wiles implies that 
\begin{equation}\label{Wiles formula}\begin{split}\dim H^1_{\cL}(\Q_S/\Q, C)-\dim H^1_{\cL^\perp}(\Q_S/\Q, C^*)&= \dim H^0(\Q, C)-\dim H^0(\Q, C^*) \\ 
+& \sum_{p\in S\cup\{\infty\}} \left(\dim \cL_p- \dim H^0(\Q_p, C)\right),
\end{split}\end{equation}
where it is understood that $\cL_\infty=0$. A variation of the Poitou--Tate sequence gives the following:
\begin{equation}\label{PT les}
    \begin{split}
        0 &\rightarrow  H^1_{\cL}(\Q_S/\Q, C)\rightarrow H^1(\Q_{S}/\Q, C)\rightarrow \bigoplus_{p\in S} \left(\frac{H^1(\Q_p, C)}{\cL_p}\right) \\
        & \rightarrow H^1_{\cL^\perp}(\Q_S/\Q, C^*)^\vee\rightarrow H^2(\Q_S/\Q, M)\rightarrow \bigoplus_{p\in S} H^2(\Q_p, C),
    \end{split}
\end{equation}
see for instance, \cite[p.~555, l.7]{tayloricosahedral}. We set $I_p$ to be the inertia subgroup of $\op{G}_{\Q_p}$ and given a $\op{G}_{\Q_p}$-module $M$, set:
\[H^1_{\op{nr}}(\Q_p, M):=\op{image}\{H^1(\op{G}_{\Q_p}/\op{I}_p, M^{\op{I}_p})\xrightarrow{\op{inf}} H^1(\Q_p, M)\}.\]
Consider the Selmer condition $\cL^{\op{nr}}$ where $\cL_p^{\op{nr}}:=H^1_{\op{nr}}(\Q_p, C)$ for all $p\in S$. We note that \begin{equation}\label{unram local dimension}\dim \cL_p^{\op{nr}}=\dim H^1(\widehat{\Z}, C^{\op{I}_p})=\dim H^0(\widehat{\Z}, C^{\op{I}_p})=\dim H^0(\Q_p, C),\end{equation}where we identify $\op{G}_{\Q_p}/\op{I}_p$ with $\widehat{\Z}$ (see also \cite[Lemma 3]{raviFM}). 
\begin{lemma}\label{Selmer and dual Selmer vanishing}
    With respect to notation above, suppose that the action of $\op{G}_{\Q}$ on $C$ is trivial. Then the following assertions hold:
    \begin{enumerate}
        \item $\dim H^1_{\cL^{\op{nr}}}(\Q_S/\Q, C)=\dim H^1_{\cL^{\op{nr}\perp}}(\Q_S/\Q, C^*)=0$.
        \item The map 
        \begin{equation}\label{Selmer map iso} H^1(\Q_S/\Q, C)\longrightarrow \bigoplus_{p\in S} \frac{H^1(\Q_p, C)}{H^1_{\op{nr}}(\Q_p, C)}\end{equation} is an isomorphism. 
    \end{enumerate}
\end{lemma}
\begin{proof}
    Since the Galois action on $C$ is trivial, $H^1(\Q_S/\Q, C)$ consists of homomorphisms $f: \op{G}_{\Q}\rightarrow C$ which are unramified away from $S$, and $H^1_{\cL^{\op{nr}}}(\Q_S/\Q, C)$ is the subset of homomorphisms which are unramified at all primes. Such unramified homomorphisms cut out unramified abelian extensions of $\Q$. Since there are no proper abelian unramified extensions of $\Q$, it follows that $H^1_{\cL^{\op{nr}}}(\Q_S/\Q, C)=0$. Thus, $\dim H^1_{\cL^{\op{nr}}}(\Q_S/\Q, C)=0$. Recall from \eqref{unram local dimension} that for $p\in S$, 
    \[\dim \cL_p^{\op{nr}}-\dim H^0(\Q_p, C)=0.\]
    We have that 
    \[\dim \cL_\infty^{\op{nr}}-\dim H^0(\Q_\infty, C)=-\dim C.\]
    Note that since the action on $C$ is the trivial one, 
    \[\dim H^0(\Q, C)=\dim C\text{ and } \dim H^0(\Q, C^*)=0. \]
    Thus, from \eqref{Wiles formula}, 
    \[\dim H^1_{\cL^{\op{nr}}}(\Q_S/\Q, C)-\dim H^1_{\cL^{\op{nr},\perp}}(\Q_S/\Q, C^*)=\dim C-0+\sum_{p\in S} 0 -\dim C=0.\] Thus, we deduce that 
    \[\dim H^1_{\cL^{\op{nr}\perp}}(\Q_S/\Q, C^*)=\dim H^1_{\cL^{\op{nr}}}(\Q_S/\Q, C)=0.\] This completes the proof of (1). Part (2) then follows from (1) and the exact sequence \eqref{PT les}.
\end{proof}

\begin{proposition}\label{step 2 propn}
    With respect to notation above, there exists $\widetilde{L}/L/\Q$ solving the embedding problem such that $\widetilde{L}$ is ramified at the same set of primes as $L$.
\end{proposition}
\begin{proof}
    It follows from Proposition \ref{global embedding problem} that there exists $\widetilde{L}/L/\Q$ solving the embedding problem. Let $\varphi:\op{G}_\Q\rightarrow G$ and $\widetilde{\varphi}:\op{G}_\Q\rightarrow \widetilde{G}$ be the homomorphisms corresponding to $L$ and $\widetilde{L}$ respectively. Let $\Sigma$ (resp. $S$) be the set of primes which ramify in $L$ (resp. $\widetilde{L}$). We modify $\widetilde{\varphi}$ by a class $f\in H^1(\op{G}_{\Q}, \Z/\ell \Z)$ such that $\widetilde{\psi}:=\widetilde{\varphi}f$ is unramified at all primes $p\notin \Sigma$. For each prime $p\in S$ we can choose a lift $\widetilde{\psi}_p:\op{G}_{\Q_p}\rightarrow \widetilde{G}_p$ of $\varphi_p$ such that for $p\in S\setminus \Sigma$, $\widetilde{\psi}_p$ is unramified. Such a lift can be constructed as in the proof of Proposition \ref{local extension is solvable}. For $p\in S\setminus \Sigma$, let $f_p\in H^1(\Q_{p}, \Z/\ell\Z)$ be such that $\widetilde{\psi}_p=\widetilde{\varphi}_p f_p$. On the other hand, for $p\in \Sigma$, set $f_p:=0$. Consider the tuple of elements 
    \[(f_p)_{p\in S}\in \left(\bigoplus_{p\in S} \frac{H^1(\Q_p, \Z/\ell\Z)}{H_{\op{nr}}^1(\Q_p, \Z/\ell\Z)}\right).\] By part 2 of Lemma \ref{Selmer and dual Selmer vanishing}, there exists $f\in H^1(\Q_S/\Q, \Z/\ell\Z)$ which restricts to $(f_p)_{p\in S}$. Note that by construction, $f_{|\op{I}_p}=f_{p|\op{I}_p}$ for all $p\in S$. It follows that $\widetilde{\psi}:=\widetilde{\varphi} f$ is unramified at all primes $p\in S\setminus \Sigma$. Let $\widetilde{L}'$ be the extension corresponding to $\widetilde{\psi}$, then $\widetilde{L}'/L/\Q$ is a solution to the embedding problem which is unramified away from $\Sigma$. 
\end{proof}

Let $F$ be the $\Z/\ell^{N-1}\Z$-extension of $\Q$ which is contained in $\Q(\mu_{\ell^N})$. Note that $\Q(\mu_{\ell^N})=F\cdot \Q(\mu_\ell)$. 
\begin{lemma}\label{linear disj lemma}
The following assertions hold:
\begin{enumerate}
    \item[(i)] The fields $F$ and $L$ are linearly disjoint,
    \item[(ii)] $FL$ and $\Q(\mu_\ell, \sqrt[\ell]{p_1}, \dots, \sqrt[\ell]{p_m})$ are linearly disjoint. 
\end{enumerate}
\end{lemma}

\begin{proof}
    The prime $\ell$ is unramified in $L$ and is totally ramified in $F$. Therefore, $F$ and $L$ are linearly disjoint. The Galois group $\op{Gal}(\Q(\mu_\ell, \sqrt[\ell]{p_1}, \dots, \sqrt[\ell]{p_m})/\Q)$ is isomorphic to $(\Z/\ell \Z)^m\rtimes (\Z/\ell\Z)^\times$, where multiplication is given by $(x,y)\cdot(x',y')=(x+yx', yy')$. It is easy to see that this group has no quotient of order $\ell$. On the other hand, $FL$ is an $\ell$-group. Hence we find that $FL \, \cap \,  \Q(\mu_\ell, \sqrt[\ell]{p_1}, \dots, \sqrt[\ell]{p_m})=\Q$.
\end{proof}

\begin{proposition}\label{constructing tilde L satisfying SN}
    There exists $\widetilde{L}/L/\Q$ solving the embedding problem such that:
    \begin{enumerate}
        \item $\widetilde{L}$ satisfies $(\mathfrak{S}_N)$,
        \item $\widetilde{L}$ is ramified at the primes in $\Sigma$ and one additional prime $q$. The prime $q$ can be chosen from a certain set of primes which has positive density.
    \end{enumerate}
     
\end{proposition}
\begin{proof}
From Proposition \ref{step 2 propn}, we obtain a solution \(\widetilde{L}\) to the embedding problem for \(\widetilde{G}\), which is ramified at precisely the same set of primes \(\Sigma = \{p_1, \dots, p_m\}\) as \(L\). Each prime \(p_i\) satisfies \(p_i \equiv 1 \pmod{\ell^N}\), and the decomposition group \(G_{p_i}\) coincides with the inertia group \(\operatorname{I}_{p_i}\) at \(p_i\). Since \(p_i\) is tamely ramified in \(L\), the inertia group \(\operatorname{I}_{p_i}\) is cyclic, and consequently, \(G_{p_i}\) is also cyclic for each \(i = 1, \dots, m\).

Now, fix \(p = p_i\) and let $I_p'$ (resp. $G_p'$) is the inertia (resp. decomposition) group of $p$ in $\widetilde{L}$. On the other hand, we set $\widetilde{I}_p$ (resp. $\widetilde{G}_p$) to be the inverse image of $I_p$ (resp. $G_p$) with respect to the map $\pi: \widetilde{G}\rightarrow G$. Since $I_p=G_p$, it follows that $\widetilde{I}_p=\widetilde{G}_p$. Since the ramification of $p$ is tame, the inertia groups $I_p$ and $I_p'$ are cyclic. Suppose that \(I_p \simeq \mathbb{Z}/\ell^\alpha\mathbb{Z}\). From the proof of Proposition \ref{local extension is solvable}, $\widetilde{G}_p$ is abelian, hence so is $\widetilde{I}_p$. Therefore, there are two possibilities for \(\widetilde{I}_p\).
\begin{description}
    \item[Case 1. \(\widetilde{I}_p \simeq \mathbb{Z}/\ell^\alpha\mathbb{Z} \times \mathbb{Z}/\ell\mathbb{Z}\)] Since $I_p'$ is cyclic and the quotient map from $I_p'$ to $I_p$ is surjective, it follows that $I_p'\simeq I_p$. On the other hand, $G_p'\subseteq \widetilde{G}_p=\widetilde{I}_p$ and thus $G_p'/I_p'$ is either $1$ or $\Z/\ell \Z$. There is an inclusion
    \[G_p'/I_p'\hookrightarrow \widetilde{G}_p/I_p'\xrightarrow{\sim} \Z/\ell \Z\] with respect to which the Frobenius element \(\sigma_p\) maps to some \(c_p \in \mathbb{Z}/\ell\mathbb{Z}\). If \(c_p = 1\), then \(G_p' = I_p'\) and the condition \((\mathfrak{S}_N)\) is satisfied at \(p\). 
    \item[Case 2. \(\widetilde{I}_p \simeq \mathbb{Z}/\ell^{\alpha+1}\mathbb{Z}\)] Since the map $I_p'\rightarrow I_p$ is surjective, it follows that $I_p'=\widetilde{I}_p=\widetilde{G}_p$. This forces $I_p'=G_p'$. In this case, we simply set $c_p:=1$.
\end{description}
If \(c_p = 1\) for all the primes $p$ in Case 1 above, then \(\widetilde{L}\) satisfies property \((\mathfrak{S}_N)\). Otherwise, we must modify \(\widetilde{L}\). Define $X$ as the set of all primes ramified in \(L\) such that $\widetilde{I}_p\simeq I_p\times \Z/\ell \Z$. \\

  Given a prime $q\equiv 1\pmod{\ell^N}$, there exists a unique surjective Galois character \[\chi^{(q)}:\op{Gal}(\Q(\mu_q)/\Q)\rightarrow \Z/\ell\Z.\] Identifying $\op{Gal}(\Q(\mu_q)/\Q)$ with $(\Z/q\Z)^\times$, and the $\ell$-torsion in $(\Z/q\Z)^\times$ with $\Z/\ell\Z$, $\chi^{(q)}$ is given by $\chi^{(q)}(a):=a^{(q-1)/\ell}$. Set \(\mathbb{L} := L\left(\mu_{\ell^N}\cup \{\sqrt[\ell]{p}| p
\in X\}\right)\) and let \(\Delta := \operatorname{Gal}(\mathbb{L}/\mathbb{Q})\). There exists \(\sigma\in \Delta\) such that, for any prime \(q\) unramified in \(\mathbb{L}\) with \(\sigma_q =\sigma\), the following conditions hold:

\begin{enumerate}
    \item \(q \equiv 1 \pmod{\ell^N}\),
    \item \(q\) splits completely in \(L/\mathbb{Q}\),
    \item \(\chi^{(q)}(p) = c_{p}\) for all $p\in X$.
\end{enumerate}
More precisely, $\sigma\in \op{Gal}(\mathbb{L}/\Q)$ is the element defined by:
\begin{itemize}
    \item $\sigma_{|FL}=1$, 
    \item $\sigma(\zeta_\ell)=\zeta_\ell$, 
    \item $\sigma(\sqrt[\ell]{p})=c_p \sqrt[\ell]{p}$ for all primes $p\in X$.
\end{itemize}
Such an element exists since $FL\cap \Q\left(\mu_{\ell^N}\cup \{\sqrt[\ell]{p}| p
\in X\}\right)=\Q$ by Lemma \ref{linear disj lemma}. Let \(\widetilde{\varphi} : \operatorname{G}_{\mathbb{Q}} \to \widetilde{G}\) correspond to \(\widetilde{L}\), and define \(\widetilde{\psi} := \widetilde{\varphi} (\chi^{(q)})^{-1}\), with the corresponding extension denoted \(\widetilde{L}'/L/\mathbb{Q}\). This modified extension solves the embedding problem for \(\widetilde{G}\) while ensuring that property \((\mathfrak{S}_N)\) is satisfied.

\end{proof}

\begin{theorem}
    Let $G$ be an $\ell$-group. Then there exist infinitely many Galois extensions $L/\Q$ with $\op{Gal}(L/\Q)\simeq G$. 
\end{theorem}
\begin{proof}
    Set $N$ to be such that $\# G=\ell^N$. Recall the filtration on $G$ by subgroups $G_i$ as in \eqref{filtration on G}. By induction on $i$, we construct Galois extensions
\[\Q=L_0\subset L_1\subset L_2\subset \dots L_{n-1}\subset L_n=L\] such that $\op{Gal}(L_i/\Q)\simeq G/G_i$ and $L_i/\Q$ satisfies the condition $(\mathfrak{S}_N)$. Assume that $L_i/\Q$ can be constructed so that $\op{Gal}(L_i/\Q)\simeq G/G_i$ and $(\mathfrak{S}_N)$ is satisfied. If the exact sequence
\[1\rightarrow \Z/\ell \Z\rightarrow G/G_{i+1}\rightarrow G/G_i\rightarrow 1\] is split, then the construction of $L_{i+1}/L_i/\Q$ follows from the argument in section \ref{split-case-proof-subsubsection}. On the other hand, if the above exact sequence is non-split, then, the construction of $L_{i+1}/L_i/\Q$ can be obtained by Proposition \ref{constructing tilde L satisfying SN} in section \ref{non-split-case-proof-subsubsection}. The construction involves an application of the Chebotarev density theorem in choosing the prime $q$. In particular, there are infinitely many choices of $q$, and hence, infinitely many choices of $L_{i+1}$. This completes the proof.
\end{proof}
\section{Selmer and diophantine stability results}\label{s 3}
\subsection{Selmer groups and diophantine stability in cyclic $\ell$-extensions}
\par We summarize and recall the results in \cite[section 2]{pathakray}. Throughout this section, $\ell$ is a fixed odd rational prime number and $L/K$ a Galois extension of number fields with $\op{Gal}(L/K)=G\simeq \Z/\ell\Z$. Let $M$ be a finite dimensional $\F_\ell$ vector space equipped with the structure of a $G$-module. Define \( M^G := \{m \in M \mid gm = m \text{ for all } g \in G\} \) and \( M_G := M / \langle (g-1)m \mid g \in G, m \in M \rangle \). The following lemma is useful.

\begin{lemma}\label{NSW lemma}
If \( M^G = 0 \) or \( M_G = 0 \), then \( M = 0 \).
\end{lemma}

\begin{proof}
See \cite[Proposition 1.6.12]{NSW}.
\end{proof}Let $E$ be a fixed elliptic curve defined over $K$. Given a finite extension $F/K$, denote by $E(F)$ (resp. $\Sh(E/F)$) the \emph{Mordell--Weil group} (resp. \emph{Tate--Shafarevich group}) of $E$ over $F$. Fix an algebraic closure \(\overline{L}\) of \(L\) and let \(\operatorname{G}_L = \operatorname{Gal}(\overline{L}/L)\) denote the absolute Galois group of \(L\). Let \(\Omega_L\) be the set of all non-archimedean places of \(L\), and for each \( v \in \Omega_L \), let \(L_v\) denote the completion of \(L\) at \(v\). Choose an algebraic closure \(\overline{L}_v\) of \(L_v\) and fix an embedding \(\iota_v: \overline{L} \hookrightarrow \overline{L}_v\). This choice induces an embedding of Galois groups \(\iota_v^*: \operatorname{G}_{L_v} \hookrightarrow \operatorname{G}_L\), where we set \(\operatorname{G}_{L_v} := \operatorname{Gal}(\overline{L}_v / L_v)\). For notational convenience, we write  
\[
H^i(L, \cdot) := H^i(\operatorname{G}_L, \cdot), \quad H^i(L_v, \cdot) := H^i(\operatorname{G}_{L_v}, \cdot)
\]  
for all \( i \geq 0 \) and \( v \in \Omega_L \). Given a global cohomology class \( f \), we denote its restriction to \(\operatorname{G}_{L_v}\) by \(\operatorname{res}_v(f)\).\\

Given a prime $v$ of $L$ and a natural number $n$, let \[\kappa_{E,v}^{(n)}:\frac{E(L_v)}{\ell^n E(L_v)}\hookrightarrow H^1(L_v, E[\ell^n])\] denote the \emph{Kummer map}. Passing to the direct limit, we get the map:
\[\kappa_{E,v}:E(L_v)\otimes \left(\Q_\ell/\Z_\ell\right)\hookrightarrow H^1(L_v, E[\ell^\infty]).\]
For $n\in \Z_{\geq 1}\cup \{\infty\}$, the Selmer group associated with the pair \((E, L)\) and the prime power \( \ell^n \) is defined as  
\[
\operatorname{Sel}_{\ell^n}(E/L) := \ker \left( H^1(L, E[\ell^n]) \to \prod_{v \in \Omega_L} H^1(L_v, E)[\ell^n] \right).
\]  
Equivalently, \(\operatorname{Sel}_{\ell^n}(E/L)\) consists of cohomology classes \( f \in H^1(L, E[\ell^n]) \) whose restrictions to every completion \( L_v \), for \( v \in \Omega_L \), lie in the image of the local Kummer map \( \kappa_{E,v}^{(n)} \). There is a natural short exact sequence relating the Mordell--Weil, Selmer and Tate--Shafarevich groups:
\[
0 \to E(L) \otimes \mathbb{Q}_\ell/\mathbb{Z}_\ell \to \operatorname{Sel}_{\ell^\infty}(E/L) \to \Sh(E/L)[\ell^\infty] \to 0.
\]

\begin{definition}\label{choice of S} Let $S$ be the finite set of primes $v$ of $K$ such that at least one of the following conditions is satisfied:
 \begin{itemize}
        \item $v$ is ramified in $L$, 
        \item $E$ has bad reduction at $v$, 
        \item $v$ divides $\ell$.
    \end{itemize}
\end{definition}  

We denote by \( S(L) \) the set of primes of \( L \) lying above those in \( S \). Let \( L_S \) be the maximal subextension of \( \overline{L}/L \) in which all non-archimedean primes \( v \notin S \) remain unramified. In fact, the \( \ell\)-Selmer group of \( E \) over \( L \) is then given by
\[
\op{Sel}_\ell(E/L) := \ker\left(H^1(L_S/L, E[\ell]) \xrightarrow{\Phi_{S,L}} \bigoplus_{w\in S(L)} H^1(L_w, E)[\ell] \right),
\]
where \( \Phi_{S,L} \) is the direct sum of restriction maps at primes \( w \in S(L) \). Given a vector space $V$ over $\F_\ell$, set $V^\vee$ to denote its dual. The Cassels--Poitou--Tate sequence then gives the following:
\[
0 \to \op{Sel}_\ell(E/K) \to H^1(K_S/K, E[\ell]) \xrightarrow{\Phi_{S,K}} \bigoplus_{v\in S} H^1(K_v, E)[\ell] \to \op{Sel}_\ell(E/K)^\vee \to H^2(K_S/K, E[\ell]) \to \dots,
\]
cf. \cite[p.8]{GCEC}. For a prime \( v \) of \( K \), define the restriction map
\[
\gamma_v: H^1(K_v, E)[\ell] \to \bigoplus_{w \mid v} H^1(L_w, E)[\ell],
\]
where \( w \) runs over primes of \( L \) above \( v \). Assume that \( \op{Sel}_\ell(E/K) = 0 \), and note that in this case the map \( \Phi_{S,K} \) is surjective. Note that there is a short exact sequence 
\[0\rightarrow E(K)\otimes \Z/\ell \Z\rightarrow \op{Sel}_\ell(E/K)\rightarrow \Sh(E/K)[\ell]\rightarrow 0. \]  As $\op{Sel}_\ell(E/K)=0$, we have that \( E(K) \otimes \mathbb{Z}/\ell \mathbb{Z} = 0 \), implying \( E(K)[\ell] = 0 \) and \( \op{rank} E(K) = 0 \). Thus we have a commutative diagram:
\begin{equation}\label{main commutative diagram}
\begin{tikzcd}[column sep = small, row sep = large]
& 0 = \op{Sel}_\ell(E/K) \arrow{r}\arrow{d}{\alpha} & H^1(K_S/K, E[\ell])\arrow{r}{\Phi_{S,K}} \arrow{d}{\beta} & \bigoplus_{v\in S} H^1(K_v, E)[\ell] \arrow{d}{\gamma} \arrow{r} & 0 \\
0\arrow{r} & \op{Sel}_\ell(E/L)^G\arrow{r} & H^1(K_S/L, E[\ell])^G\arrow{r}  &\left(\bigoplus_{w\in S(L)} H^1(L_w, E)[\ell]\right)^G,
\end{tikzcd}
\end{equation}
where \( \gamma = \bigoplus_{v \in S} \gamma_v \).

\begin{proposition}\label{vanishing of selmer propn}
Suppose \( \op{Sel}_\ell(E/K) = 0 \). Let \( S \) be as in Definition \ref{choice of S}. Then:
\begin{enumerate}
    \item[(i)] \( \dim_{\F_\ell}(\op{Sel}_\ell(E/L)^G) = \sum_{v \in S} \dim_{\F_\ell}(\ker \gamma_v) \).
    \item[(ii)] \( \op{Sel}_\ell(E/L) = 0 \) if and only if \( \gamma_v \) is injective for all \( v \in S \).
\end{enumerate}
\end{proposition}

\begin{proof}
Consider the commutative diagram \eqref{main commutative diagram}. The inflation-restriction sequence gives
\[
\ker \beta \simeq H^1(L/K, E(L)[\ell]), \quad \coker \beta \subseteq H^2(L/K, E(L)[\ell]).
\]
As \( L/K \) is a \( p \)-extension and \( E(L)[\ell]^G = E(K)[\ell] = 0 \), Lemma \ref{NSW lemma} yields \( E(L)[\ell] = 0 \), so \( \ker \beta = \coker \beta = 0 \). Thus, \( \op{Sel}_\ell(E/L)^G \simeq \ker \gamma \), proving (i). The claim in (ii) follows from Lemma \ref{NSW lemma}.
\end{proof}

Given a prime $v$ of $K$, let $k_v$ be the residue field of $K$ at $v$ and let $\widetilde{E}(k_v)$ be the group of $k_v$ points of the reduction of $E$ at $v$.
\begin{proposition}\label{basic prop for Sel(E/L)=0}
    With respect to the above notation, assume that the following conditions hold:
    \begin{enumerate}
        \item[(i)] $\op{Sel}_\ell(E/K)=0$,
        \item[(ii)] all primes of $K$ that lie above $\ell$ and the primes of bad reduction for $E$ are completely split in $L$,
        \item[(iii)] all primes $v$ of $K$ that ramify in $L$ satisfy $\widetilde{E}(k_v)[\ell]=0$.
    \end{enumerate}
Then, $\op{Sel}_\ell(E/L)=0$. 
\end{proposition}
\begin{proof}
    It follows from the assumptions above that $\gamma_\ell$ is injective for every prime $\ell\in S$ and thus by Proposition \ref{vanishing of selmer propn}, $\op{Sel}_\ell(E/L)=0$. For further details, we refer to the proof of \cite[Proposition 2.9]{pathakray}.
    \end{proof}
Let \(\rho_{E, \ell}:\op{G}_{\Q} \to \op{GL}_2(\Z/\ell\Z)\) be the Galois representation on \(E[\ell]\), and let \(\Q(E[\ell])\) be the field fixed by \(\ker \rho_{E,\ell}\), so that \(\op{Gal}(\Q(E[\ell])/\Q) \simeq \op{im} \rho_{E,\ell}\). For a prime \(p\) of good reduction for \(E\), let \(\op{G}_{\Q_p} = \op{Gal}(\overline{\Q}_p/\Q_p)\) with inertia subgroup \(\op{I}_p\), and choose a lift \(\sigma_p \in \op{G}_{\Q_p}\) of the Frobenius. The trace and determinant of \(\rho_{E,\ell}(\sigma_p)\) satisfy  
\[
\op{trace} \rho_{E,\ell}(\sigma_p) \equiv a_p(E) \bmod \ell, \quad \op{det} \rho_{E, \ell}(\sigma_p) \equiv p \bmod \ell.
\]
Setting \(K(E[\ell]) = K\cdot \Q(E[\ell])\), we introduce a key set of primes central to our results. 

\begin{definition}\label{defn of TE}Define \( \mathfrak{T}_{E,K} \) as the set of rational primes \( p \) satisfying: (a) \( p \neq \ell \), (b) \( p \) is a prime of good reduction for \( E \), (c) \( p \) is completely split in \( K(\mu_\ell) \), and (d) \(\op{trace}(\rho_{E,\ell}(\sigma_p)) \neq 2 \).
\end{definition}
\begin{remark}\label{TE remark}For $p\in \mathfrak{T}_{E,K}$, $\widetilde{E}(k_v)[\ell]=0$ for any of the primes of $K$ such that $v|\ell$ (cf. \cite[Lemma 2.11]{pathakray}). 
\end{remark}

\begin{lemma}\label{density lemma}
   If \(\rho_{E, \ell}\) is surjective and \(K(\mu_\ell) \cap \Q(E[\ell]) = \Q(\mu_\ell)\), then \(\mathfrak{T}_{E,K}\) has natural density  
\[
\mathfrak{d}(\mathfrak{T}_{E,K}) = \frac{\ell^2 - \ell - 1}{[K(\mu_\ell):\Q(\mu_\ell)]\,(\ell^2-1)(\ell-1)}.
\]
In particular, \(\mathfrak{T}_{E,K}\) is infinite.
\end{lemma}
\begin{proof}
See the proof of \cite[Lemma 2.12]{pathakray}.
\end{proof}

\subsection{Main results}
\par Let $\ell\geq 5$ be a prime number and $E$ be an elliptic curve over $\Q$. We note the following recurring observation.
\begin{lemma}\label{PSL2-lemma}
Assume that $\rho_{E,\ell}$ is surjective and let $F/\Q(\mu_\ell)$ be a Galois extension for which $\op{Gal}(F/\Q(\mu_\ell))$ is an $\ell$-group. Then, 
\begin{equation*}
    F \cap \Q(E[\ell]) = \Q(\mu_{\ell}).
\end{equation*}
\end{lemma}
\begin{proof}
Since $\rho_{E,\ell}$ is surjective, it induces an isomorphism \[\varrho: \op{Gal}\big(\Q(E[\ell])/\Q(\mu_{\ell}) \big) \xrightarrow{\sim} \op{SL}_2(\F_{\ell}).\] Set $F_0 := F \, \cap \, \Q(E[\ell])$ and \[N:=\varrho\left(\op{Gal}(\Q(E[\ell])/F_0)\right)  \subseteq \op{SL}_2(\F_{\ell}).\] Consider the field diagram: 
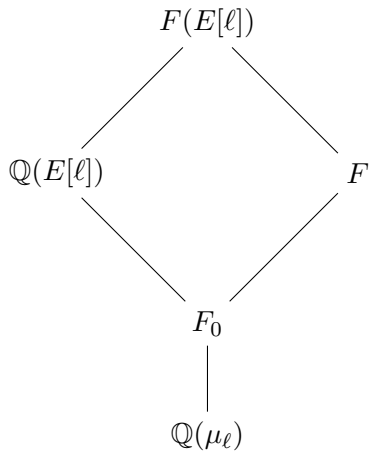
\begin{center}
\begin{tikzpicture}
    \node (Q0) at (0,0) {$\Q(\mu_\ell)$};
    \node (Q1) at (0,1.5) {$F_0$};
    \node (Q2) at (2,3.5) {$F$};
    \node (Q3) at (-2,3.5) {$\Q(E[\ell])$};
    \node (Q4) at (0,5.5) {$F(E[\ell])$};
     \draw (Q0)--(Q1);
    \draw (Q1)--(Q2) node [pos=0.7, below,inner sep=0.25cm] {};
    \draw (Q1)--(Q3) node [pos=0.7, below,inner sep=0.25cm] {};
    \draw (Q2)--(Q4);
    \draw (Q3)--(Q4);
    \end{tikzpicture}
\end{center}

Since $F_0$ is a Galois $\ell$-extension of $\Q(\mu_\ell)$, $N$ is a normal subgroup of $\op{SL}_2(\F_{\ell})$ and we have that
\begin{equation*}
    [\op{SL}_2(\F_{\ell}): N] = [F_0 : \Q(\mu_{\ell})] = \ell^k,
\end{equation*}
for some $ 1 \leq \ell^k \leq [F:\Q(\mu_{\ell})]$. For primes $\ell\geq 5$, Galois \cite[p.\ 412]{galois1846works} showed that $\op{PSL}_2(\F_\ell)$ is simple. Therefore, $\op{SL}_2(\F_\ell)$ does not contain a proper normal subgroup $N$ with odd index $[\op{SL}_2(\F_\ell):N]$. It thus follows that $\ell^k = 1$ and $F_0 = \Q(\mu_{\ell})$.
\end{proof}

For the discussion below, assume the following:
\begin{itemize}
    \item $\rho_{E,\ell}$ is surjective, and
    \item $\op{Sel}_\ell(E/\Q)=0$.
\end{itemize}
For ease of notation, we set $\mathfrak{T}_E:=\mathfrak{T}_{E,\Q}$. Let $\Sigma$ be a finite set of primes containing $\ell$ and the primes at which $E$ has bad reduction. The result below is used in conjunction with Proposition \ref{basic prop for Sel(E/L)=0} to prove Theorem \ref{thm of section 3}, which is the main result of this section.
\begin{proposition}\label{technical propn}
    Let $G$ be an $\ell$-group with $\#G=\ell^n$ and let $N\geq n$. There are infinitely many Galois extensions $L/\Q$ such that:
    \begin{enumerate}[label=(\roman*)]
        \item $\op{Gal}(L/\Q)\simeq G$, 
        \item if $p_1, \dots, p_m$ are all the primes that are ramified in $L$, then $m\leq n$, $p_1, \dots, p_m\in \mathfrak{T}_E\setminus \Sigma$, $p_i\equiv 1\pmod{\ell^N}$ and the inertial degree of $p_i$ in $L$ is $1$ for $i=1, \dots, m$,
        \item all primes $q\in \Sigma$ are completely split in $L$. 
    \end{enumerate}
\end{proposition}
\begin{proof}
    We prove the result by induction on $n$. First we consider the case when $n=1$, i.e., when $G\simeq \Z/\ell\Z$. Given a prime $p\equiv 1\pmod{\ell}$, let $L^p$ be the unique $\Z/\ell\Z$-extension of $\Q$ which is contained in $\Q(\mu_p)$. We show that there is a positive density of primes $p$ for which:
    \begin{enumerate}
        \item $p\in \mathfrak{T}_E\backslash \Sigma$ \item $p\equiv 1\pmod{\ell^N}$
        \item  all primes $q\in \Sigma$ split completely in $L^p$.
    \end{enumerate} 
     For such a $p$, $p$ is totally ramified in $L^p$ and thus the inertial degree of $p$ in $L^p$ is $1$. Therefore, $L^p$ satisfies the conditions (i), (ii), (iii).  We show that $p$ as above is determined by certain non-empty Chebotarev conditions. In other words, there exists a Galois extension $\mathcal{F}/\Q$ and a conjugation invariant non-empty subset $\mathcal{S}\subseteq \op{Gal}(\mathcal{F}/\Q)$, such that a prime $p$ which is unramified in $\mathcal{F}$ satisfies (1), (2), (3) if and only if the Frobenius element $\sigma_p\in \mathcal{S}$. The condition that $p\equiv 1\pmod{\ell^N}$ is equivalent to the requirement that $p$ splits completely in $\Q(\mu_{\ell^N})$, i.e., $\sigma_p=1$. The primes $q\in \Sigma$ are required to split in $L^p$. In other words, $q$ is an $\ell$-th power modulo $p$. Thus, $p$ is required to split completely in $\cF':=\Q(\mu_{\ell^N}\cup\{\sqrt[\ell]{q} \mid q\in \Sigma\})$. On the other hand, by Lemma \ref{PSL2-lemma}, $\cF'\cap \Q(E[\ell])=\Q(\mu_\ell)$. Setting $\cF:=\cF'\cdot \Q(E[\ell])=\Q(E[\ell]\cup \mu_{\ell^N}\cup\{\sqrt[\ell]{q} \mid q\in \Sigma\})$, one finds that
     \[\op{Gal}\left(\cF/\Q(\mu_\ell)\right)\simeq \op{Gal}\left(\cF'/\Q(\mu_\ell)\right)\times \op{Gal}\left(\Q(E[\ell])/\Q(\mu_\ell)\right).\] Let $\mathcal{S}:=\mathcal{S}_1\times \mathcal{S}_2$ where $\mathcal{S}_1=\{1\}\subset \op{Gal}(\cF'/\Q(\mu_\ell))$ and $\mathcal{S}_2\subset \op{Gal}(\Q(E[\ell])/\Q(\mu_\ell))$ consisting of $\sigma$ such that $\op{trace}\rho_{E,\ell}(\sigma)\not \equiv 2\pmod{\ell}$. It is clear that if $p\notin \Sigma$ is a prime which is unramified in $\cF$ and $\sigma_p\in \mathcal{S}$, then  the conditions (1)--(3) are satisfied. By the Chebotarev density theorem, this set of primes $p$ has positive density. 
    \par For $n\geq 2$ by induction hypothesis, suppose that we are given a central extension
    \[1\rightarrow \Z/\ell \Z\rightarrow G\rightarrow \overline{G}\rightarrow 1\] and that there exists a Galois extension $L/\Q$ of degree $\ell^{n-1}$ with $\op{Gal}(L/\Q)\simeq \overline{G}$ such that (i)--(iii) are satisfied. Let $\{p_1, \dots, p_m\}$ (with $m\leq n-1$) be the primes that ramify in $L$. Then we show that there exists $\widetilde{L}/L/\Q$ such that $\op{Gal}(\widetilde{L}/\Q)\simeq G$ satisfying (i)--(iii), and $\widetilde{L}$ is ramified at $\{p_1, \dots, p_m, p_{m+1}\}$ where $p_{m+1}$ is a prime which belongs to a set defined by non-empty Chebotarev conditions, thus completing the inductive step.
    
    \par By Proposition \ref{step 2 propn}, there exists $\widetilde{L}_0/L/\Q$ such that $\op{Gal}(\widetilde{L}_0/\Q)\simeq G$ and the only primes which ramify in $\widetilde{L}_0$ are $p_1, \dots, p_m$. We shall modify $\widetilde{L}_0$ (as in Definition \ref{defn of twist}) to get $\widetilde{L}/L/\Q$ as above. Since $L$ satisfies condition (ii), each of the primes $p_i$ has inertial degree $1$ in $L$. Let $c_{p_i}\in \Z/\ell \Z$ be as defined in the proof of Proposition \ref{constructing tilde L satisfying SN}. Since $L$ satisfies condition (iii), for each prime $w\in \Sigma$, there is an element $c_w\in \Z/\ell\Z$ such that the Frobenius at $w$ maps to $c_w$. Now let $p$ be a prime such that:
    \begin{enumerate}[label=(\alph*)]
        \item $p\in \mathfrak{T}_E\setminus(\Sigma\cup \{p_1, \dots, p_m\})$, 
        \item $p$ splits in $L(\mu_{\ell^N})$, 
        \item for each prime $v\in \Sigma\cup \{p_1, \dots, p_m\}$, $\sigma_v$ maps to $c_v$ in $\Z/\ell\Z$ (via the natural map $\op{G}_{\Q_v}\rightarrow G$). 
    \end{enumerate}
    These conditions are determined by non-empty Chebotarev sets. Condition (a) is determined in the field extension $\Q(E[\ell])/\Q(\mu_\ell)$, and (b),(c) are determined in $\cF':=L(\mu_{\ell^N}, \sqrt[\ell]{v}\mid v\in \Sigma \cup \{p_1, \dots, p_m\} )$. Let $\mathcal{S}_1\subset \op{Gal}(\cF'/\Q(\mu_\ell))$ be the Chebotarev set which corresponds to conditions (b) and (c), and $\mathcal{S}_2 \subset\op{Gal}(\Q(E[\ell])/\Q(\mu_\ell))$ corresponds to condition (a). Note that $\cF'$ is an $\ell$-extension of $\Q(\mu_\ell)$.  Lemma \ref{PSL2-lemma} thus implies that $\cF'\cap \Q(E[\ell])=\Q(\mu_\ell)$. Therefore, 
    \[\op{Gal}\left(\cF/\Q(\mu_\ell)\right)=\op{Gal}\left(\cF'/\Q(\mu_\ell)\right)\times \op{Gal}\left(\Q(E[\ell])/\Q(\mu_\ell)\right)\] and let $\cS:=\cS_{1} \times \cS_{2}$. By construction, $\cS\subset \op{Gal}(\cF/\Q)$ is the non-empty Chebotarev set which determines conditions (a)--(c). Thus, by the Chebotarev density theorem, the set of primes satisfying (a)--(c) has positive density. We choose one such prime $p$ and take $p_{m+1}:=p$. Let $\chi^{(p)}: \op{Gal}(\Q(\mu_p)/\Q)\rightarrow \Z/\ell \Z$ be the associated character and take $\widetilde{L}$ be the twist of $\widetilde{L}_0$ by $(\chi^{(p)})^{-1}$ as in Definition \ref{defn of twist}. This modification satisfies all the required conditions and completes the inductive argument. 
\end{proof}

\begin{theorem}\label{thm of section 3}
    Let $\ell\geq 5$ be a prime and $E$ be an elliptic curve over $\Q$. Assume that $\rho_{E,\ell}:\op{G}_{\Q}\rightarrow \op{GL}_2(\F_\ell)$ is surjective and that $\op{Sel}_\ell(E/\Q)=0$. Let $G$ be an $\ell$-group. Then there are infinitely many $G$-extensions $L/\Q$ for which $\op{Sel}_\ell(E/L)=0$.
\end{theorem}
\begin{proof}
 According to Proposition \ref{technical propn}, there exist infinitely many Galois extensions \( L/\Q \) with {\( \operatorname{Gal}(L/\Q) \simeq G \)} such that:
\begin{itemize}
    \item if \( p_1, \dots, p_m \) denote the primes that ramify in \( L \), then \( m \leq n \), each \( p_i \in \mathfrak{T}_E \setminus \Sigma \), \( p_i \equiv 1 \bmod{\ell^N} \), and the inertial degree of \( p_i \) in \( L \) is equal to \( 1 \) for all \( i = 1, \dots, m \);
    \item every prime \( q \in \Sigma \) is completely split in \( L \).
\end{itemize}
We fix such an extension \( L/\Q \) and filter it by a tower of intermediate fields
\[
\Q = L_0 \subset L_1 \subset L_2 \subset \cdots \subset L_{n-1} \subset L_n = L,
\]
with \( \operatorname{Gal}(L_i/\Q) \simeq G/G_i \) and $G_i$ are as in \eqref{filtration on G}. By assumption, \( \operatorname{Sel}_\ell(E/L_0) = 0 \). We prove by induction on \( i \) that \( \operatorname{Sel}_\ell(E/L_i) = 0 \), assuming the vanishing of \( \operatorname{Sel}_\ell(E/L_{i-1}) \). To do so, it suffices to verify that the extension \( L_i/L_{i-1} \) satisfies conditions (i)--(iii) of Proposition \ref{basic prop for Sel(E/L)=0}.

Since \( \operatorname{Sel}_\ell(E/L_{i-1}) = 0 \) by the inductive hypothesis, condition (i) is satisfied. Moreover, because all rational primes in \( \Sigma \) split completely in \( L \), in particular every prime of \( L_{i-1} \) lying above \( \Sigma \) splits completely in \( L_i \), verifying condition (ii). To verify condition (iii), let \( v \) be a prime of \( L_{i-1} \) that ramifies in \( L_i \). Then \( v \mid p_i \) for some \( i \), and since the inertial degree of \( p_i \) in \( L \) is \( 1 \), it follows that the residue field \( k_v \simeq \F_{p_i} \). For each \( p_i \in \mathfrak{T}_E \), we have \( \widetilde{E}(\F_{p_i})[\ell] = 0 \) by Remark \ref{TE remark}, so condition (iii) is satisfied. This completes the proof.
\end{proof}

\section{Density results and connections with Malle's conjecture}\label{s 4}
\par In this section, we fix a prime \( \ell \geq 5 \) and an elliptic curve \( E/\Q \) satisfying:
\begin{itemize}
    \item \( \operatorname{Sel}_\ell(E/\Q) = 0 \),
    \item the mod \( \ell \) Galois representation \( \rho_{E,\ell} : \operatorname{Gal}(\overline{\Q}/\Q) \to \operatorname{GL}_2(\F_\ell) \) is surjective.
\end{itemize}
Let $G$ be a finite $\ell$-group with $\# G=\ell^n$. We recall that Theorem \ref{thm of section 3} establishes the existence of infinitely many Galois extensions \( L/\Q \) such that \( \operatorname{Gal}(L/\Q) \simeq G \) and \( \operatorname{Sel}_\ell(E/L) = 0 \). In this section, we go further and obtain asymptotic lower bounds for the number of such extensions \( L/\Q \), counted by their absolute discriminant. We then compare these lower bounds with the expected asymptotics for the total number of \( G \)-extensions of \( \Q \), as predicted by the weak form of Malle's conjecture.

\subsection{Counting $\ell$-extensions by discriminant} Let $G$ be a finite $\ell$-group and $\overline{G}$ be a quotient of $G$ such that there is a central extension:
\begin{equation}\label{embedding problem section 4}1\rightarrow \Z/\ell \Z \rightarrow G\rightarrow \overline{G}\rightarrow 1.\end{equation} Let $L/\Q$ be a Galois extension with $\op{Gal}(L/\Q)\simeq \overline{G}$ and $\widetilde{L}/L/\Q$ a solution to the above embedding problem. Let $\zeta$ be a primitive $\ell$-th root of unity and let $\sigma$ be a generator of $\op{Gal}(\widetilde{L}(\mu_\ell)/\widetilde{L})$ with $\sigma(\zeta)=\zeta^q$.

\begin{proposition}
   With respect to notation above, the following assertions hold.
\begin{enumerate}
    \item[(i)] There exists $\alpha\in L(\mu_\ell)^\times$ such that $\widetilde{L}(\mu_\ell)=L(\mu_\ell, \sqrt[\ell]{\alpha})$ and $\sigma(\alpha)/\alpha^q\in \left(L(\mu_\ell)^\times\right)^q$.
    \item[(ii)] Any other solution $\widetilde{L}'/L/\Q$ to the embedding problem is given by $\widetilde{L}'=\widetilde{L}_b$, where $\widetilde{L}_b$ is the index $(\ell-1)$ subfield of $L(\mu_\ell, \sqrt[\ell]{b\alpha})$ where $b\in \Q(\mu_\ell)^\times$ and satisfies $\sigma(b)/b^q\in (\Q(\mu_\ell)^\times)^\ell$. 
\end{enumerate}

\end{proposition}
\begin{proof}
    The assertions (i) and (ii) follow from Proposition 3.3 and 3.4 of \cite{klunersmalle} respectively.
\end{proof}
Let $\Q_1$ be the unique $\Z/\ell\Z$-extension of $\Q$ which is contained in $\Q(\mu_{\ell^2})$. Let $b\in \Q(\mu_\ell)^\times$ be such that $\sigma(b)/b^q\in (\Q(\mu_\ell)^\times)^\ell$. Setting $L:=\Q$ and $\overline{G}=1$, we choose $\widetilde{L}:=\Q_1$. Then, any other cyclic $\ell$-extension of $\Q$ is obtained by twisting $\Q_1$ by an element $b$. One denotes this twist by $\Q_b$ and thus we have an indexing of all cyclic $\ell$-extensions of $\Q$.
\begin{proposition}\label{KM result}
     With respect to the notation above, the following assertions hold. 
    \begin{enumerate}
   \item[(i)] The association $\Q_b\mapsto \widetilde{L}_b$ between cyclic $\ell$-extensions of $\Q$ and solutions to the embedding problem for $L/\Q$ has finite fibers with cardinality bounded only in terms of $\overline{G}$ and $\ell$.
   \item[(ii)] There is a constant $C>0$ such that 
\[C |\Delta_{\widetilde{L}_b}|\leq |\Delta_{\Q_b}|^{\ell^{n-1}}.\]
   \end{enumerate}
\end{proposition}
\begin{proof}
    The result follows from \cite[Proposition 3.5]{klunersmalle}.
\end{proof}

Recall from Definition \ref{defn of twist} that $\widetilde{L}'=\widetilde{L}^f$ where $f\in H^1(\Q, \Z/\ell\Z)$. Identify $H^1(\Q, \Z/\ell\Z)$ with $\op{Hom}(\op{G}_\Q, \Z/\ell\Z)$. By Kummer theory, there's an isomorphism
\[\op{Hom}(\op{G}_\Q, \Z/\ell\Z)\simeq \{b\in \Q(\mu_\ell)^\times/(\Q(\mu_\ell)^\times)^\ell: \sigma(b)/b^q\in (\Q(\mu_\ell)^\times)^\ell\}.\] Let $f_b$ be the homomorphism corresponding to $b$. We have that $\widetilde{L}^{f_b}=\widetilde{L}_b$.

\subsection{An application of Wiles' formula}\label{section 4.2}
\par Recall that $E_{/\Q}$ is an elliptic curve such that $\op{Sel}_\ell(E/\Q)=0$ and $\rho_{E,\ell}$ is surjective. It follows from Proposition \ref{technical propn} that there exists $\widetilde{L}/L/\Q$ solving the embedding problem \eqref{embedding problem section 4} such that
\begin{enumerate}[label=(\roman*)]
        \item if $p_1, \dots, p_m$ are all the primes that are ramified in $\widetilde{L}$, then $m\leq n$, $p_1, \dots, p_m\in \mathfrak{T}_E\setminus \Sigma$, $p_i\equiv 1\pmod{\ell^N}$ and the inertial degree of $p_i$ in $L$ is $1$ for $i=1, \dots, m$,
        \item all primes $q\in \Sigma$ are completely split in $\widetilde{L}$. 
    \end{enumerate}
    Let $S$ be a finite set of primes in $\mathfrak{T}_{E,L}$ disjoint from $\Sigma \cup \{p_1, \dots, p_m\}$. Let $V_S$ be the space of cohomology classes $f\in H^1(\Q_S/\Q, \Z/\ell\Z)$ such that $\op{res}_q(f)=0$ for all primes $q\in \Sigma \cup \{p_1, \dots, p_m\}$. 
Setting $\widetilde{S}:=S\cup \Sigma\cup\{p_1, \dots, p_m\}$, the space $V_S=H^1_{\cL}(\Q_{\widetilde{S}}/\Q, \Z/\ell\Z)$ is determined by Selmer conditions $\{\mathcal{L}_q\}_{q\in \widetilde{S}}$ defined as follows:
\[\cL_q:=\begin{cases}
    H^1(\Q_q, \Z/\ell\Z) &\text{ if } q\in S\\
   0 &\text{ if } q\in \Sigma\cup\{p_1, \dots, p_m\}.
\end{cases}\]
Let $\overline{\chi}$ be the mod $\ell$ cyclotomic character. The dual Selmer group $V_S^\perp:=H^1_{\cL^\perp}(\Q_{\widetilde{S}}/\Q, \Z/\ell\Z(\overline{\chi}))$ is given by:
\[V_S^\perp=\op{ker}\{H^1(\Q_{\widetilde{S}}/\Q, \Z/\ell\Z(\overline{\chi}))\longrightarrow \bigoplus_{q\in S} H^1(\Q_q, \Z/\ell\Z(\overline{\chi}))\}.\]
Given a prime $q\equiv 1\mod{\ell}$, it follows from local class field theory that \[\dim H^1(\Q_q, \Z/\ell\Z)=2.\] Wiles' formula \eqref{Wiles formula} yields 
\[\begin{split}\dim V_S\geq & 1+\sum_{q\in S} \left( \dim H^1(\Q_q, \Z/\ell\Z)- \dim H^0(\Q_q, \Z/\ell\Z)\right)\\ - &\sum_{q\in \Sigma\cup\{p_1, \dots, p_m\}\cup\{\infty\}}  \dim H^0(\Q_q, \Z/\ell\Z)\\ 
=& \#S-\#\Sigma-m.
\end{split}\]

\begin{lemma}\label{L twist f lemma}
    Let $S$ be a finite set of primes in $\mathfrak{T}_{E,L}$ disjoint from $\Sigma \cup \{p_1, \dots, p_m\}$ and $f\in V_S$. Then, the following assertions hold:
    \begin{enumerate}
\item[(i)]$\op{Sel}_\ell(E/\widetilde{L}^f)=0$,
        \item[(ii)] $ |\Delta_{\widetilde{L}^f}|\leq c_2\left(\prod_{q\in S} q\right)^{\ell^{n-1}(\ell-1)}$, for an absolute constant $c_2>0$.
    \end{enumerate}
     
\end{lemma}

\begin{proof}
    Let $q_1, \dots, q_r$ be the rational primes which are ramified in $\widetilde{L}^f$. Then by construction, $q_1, \dots, q_r$ are primes in $S\cup \{p_1, \dots,p_m\}$. Since $\op{res}_{q}(f)=0$ at all primes of $q\in \Sigma$, it follows that all primes of $\Sigma$ are completely split in $\widetilde{L}^f$.  On the other hand, let $v$ be a prime of $L$ which ramifies in $\widetilde{L}^f$ and $q$ be the rational prime such that $v|q$. If $q=p_i$, then $k_v=\F_{p_i}$. If $q\in S$, then $q$ splits in $L$ and we have that $k_v=\F_q$. Thus, $\widetilde{E}(k_v)[\ell]=\widetilde{E}(\F_q)[\ell]=0$. It then follows from Proposition \ref{basic prop for Sel(E/L)=0} that 
    \[\op{Sel}_\ell(E/L)=0\Rightarrow  \op{Sel}_\ell(E/\widetilde{L}^f)=0.\] This proves part (i).

    \par Let $\Q_f$ be the cyclic $\ell$-extension of $\Q$ which is cut out by $f$ and set $\Delta_f:=\Delta_{\Q_f}$. We have that $|\Delta_f|\leq (\prod_{q\in S} q)^{\ell-1}$ and part (ii) follows from Proposition \ref{KM result}.
\end{proof}

\begin{lemma}\label{choice of S_0}
    There is a finite set of primes $S_0$ contained in $\mathfrak{T}_{E,L}$ such that $V_{S_0}^\perp=0$.
\end{lemma}
\begin{proof}
Let $\psi_1, \dots, \psi_r$ be the non-zero elements in $H^1(\Q_{\Sigma\cup \{p_1, \dots, p_m\}}/\Q, \Z/\ell \Z(\overline{\chi}))$. Let $\psi\in \{\psi_1, \dots, \psi_r\}$ and $\Omega:=\op{Gal}(\Q(\mu_\ell)/\Q)$. Note that the inflation--restriction sequence gives:
{\small
\[ 0\rightarrow H^1\left(\Omega,\Z/\ell \Z(\overline{\chi})\right)\rightarrow H^1(\Q_{\Sigma\cup \{p_1, \dots, p_m\}}/\Q, \Z/\ell \Z(\overline{\chi}))\xrightarrow{\op{res}} \op{Hom}(\Q_{\Sigma\cup \{p_1, \dots, p_m\}}/\Q(\mu_\ell), \Z/\ell \Z(\overline{\chi}))^{\Omega}.\]}
Since $\# \Omega$ is coprime to $\ell$, $H^1\left(\Omega,\Z/\ell \Z(\overline{\chi})\right)=0$ and thus the restriction map is an injection. Thus the homomorphism $\op{res}(\psi)$ is non-zero. It gives rise to a $\Z/\ell \Z$ extension $L_\psi$ over $\Q(\mu_\ell)$, which is a Galois extension of $\Q$ with 
\begin{equation}\label{semi-direct-product}
\op{Gal}(L_{\psi}/\Q)\simeq \Omega \ltimes \op{Gal}(L_{\psi}/\Q(\mu_\ell)).
\end{equation}
Here, the action of $\Omega$ on $\op{Gal}(L_{\psi}/\Q(\mu_\ell))\simeq \Z/\ell\Z$ is via the character $\overline{\chi}$. 
\par There are two possibilities for the intersection $L_{\psi}\cap L(\mu_\ell)$, namely either $L_{\psi}\cap L(\mu_\ell)=L_{\psi}$ or $L_\psi\cap L(\mu_\ell)=\Q(\mu_\ell)$. Recall that $\Omega$ acts non-trivially on $\op{Gal}(L_\psi/\Q(\mu_\ell))$. On the other hand, $\op{Gal}(L(\mu_\ell)/\Q)\simeq \Omega\times \op{Gal}(L/\Q)$. Therefore, if $L_\psi\subset L(\mu_\ell)$, then the action of $\Omega$ on $\op{Gal}(L_\psi/\Q(\mu_\ell))$ would be trivial, which contradicts \eqref{semi-direct-product}. Thus we have that $L_{\psi}\cap L(\mu_\ell)=\Q(\mu_\ell)$. Lemma \ref{PSL2-lemma} implies that $L_{\psi}\cap \Q(E[\ell])=\Q(\mu_\ell)$. The Chebotarev conditions determining $\mathfrak{T}_{E,L}$ are thus independent of the extension $L_\psi/\Q(\mu_\ell)$. There is thus a nonempty Chebotarev set contained in $\op{Gal}(L_\psi \cdot L(E[\ell])/\Q)$ determining a positive density set of primes $q\in \mathfrak{T}_{E,L}$ which are non-split in the extension $L_\psi/\Q(\mu_\ell)$. Pick one such prime $q_i$ for each $\psi_i$ for $i=1, \dots, r$. Then setting $S_0:=\{q_1, \dots, q_r\}$, we have that:
\[V_{S_0}^\perp=\op{ker}\{H^1(\Q_{S_0\cup \Sigma\cup \{p_1, \dots,p_m\}} /\Q, \Z/\ell\Z(\overline{\chi}))\longrightarrow \bigoplus_{q\in S_0} H^1(\Q_q, \Z/\ell\Z(\overline{\chi}))\}.\]
For each $\psi_i$, $\op{res}_{q_i}(\psi_i)\neq 0$ and hence $\psi_i\notin V_{S_0}^\perp$. This implies that $V_{S_0}^\perp=0$, thus proving the result. 
\end{proof}

\begin{definition}\label{defn of Z}
   For the rest of this article, we fix a choice of $S_0$ as in Lemma \ref{choice of S_0} and set $Z:=S_0\cup \Sigma \cup \{p_1, \dots, p_m\}$.  
\end{definition}

Let $T\subset\mathfrak{T}_{E,L}\setminus Z$ be a finite set of primes. Let $W_T$ be the subset of $V_{{S_0}\cup T}$ consisting of $f$ which are ramified at each of the primes in $T$. We have a natural exact sequence:

\begin{equation}\label{left exact sequence}0\rightarrow V_{S_0}\xrightarrow{\iota} V_{S_0\cup T} \xrightarrow{\pi} \bigoplus_{q\in T} \frac{H^1(\Q_q, \Z/\ell \Z)}{H^1_{\op{nr}}(\Q_q, \Z/\ell \Z)}.\end{equation}

\begin{lemma}
    The following assertions hold:
    \begin{enumerate}
        \item[(i)] The sequence \eqref{left exact sequence} above is a short exact sequence, i.e., $\pi$ is surjective;
        \item[(ii)] $\# W_T=(\ell-1)^{\# T}\#V_{S_0}$.
    \end{enumerate}
\end{lemma}
\begin{proof}
    For the proof of part (i) it suffices to show that 
    \[\dim V_{S_0\cup T}=\dim V_{S_0}+\sum_{q\in T} \dim \left(\frac{H^1(\Q_q, \Z/\ell \Z)}{H^1_{\op{nr}}(\Q_q, \Z/\ell \Z)}\right).\]
    There is a unique unramified $\Z/\ell\Z$-extension of $\Q_q$, and thus, $\dim H^1_{\op{nr}}(\Q_q, \Z/\ell \Z)=1$. Since $q\equiv 1\pmod{\ell}$ for all primes $q\in \mathfrak{T}_{E,L}$, a simple application of local class field theory shows that $\dim H^1(\Q_q, \Z/\ell \Z)=2$. It thus follows that\[\dim \left(\frac{H^1(\Q_q, \Z/\ell \Z)}{H^1_{\op{nr}}(\Q_q, \Z/\ell \Z)}\right)=1.\]
    By the choice of $S_0$ in Lemma \ref{choice of S_0} one has that $V_{S_0}^\perp=0$ and consequently, $V_{S_0\cup T}^\perp=0$ as well. Thus by Wiles' formula, 
    \[\dim V_{S_0\cup T}=\dim V_{S_0} +\sum_{q\in T} (\dim H^1(\Q_q, \Z/\ell\Z)-\dim H^0(\Q_q, \Z/\ell\Z))=\dim V_{S_0}+\#T\] and part (i) follows.
    \par Part (ii) can be derived from part (i) by noting that $W_T=\pi^{-1}(\mathcal{W})$, where \[\mathcal{W}\subset \bigoplus_{q\in T} \frac{H^1(\Q_q, \Z/\ell \Z)}{H^1_{\op{nr}}(\Q_q, \Z/\ell \Z)}\] consists of elements $(h_q)_{q\in T}$ such that $h_q\neq 0$ for all $q\in T$. We find that 
    \[\#W_T=\#V_{S_0}\times \# \mathcal{W}=\#V_{S_0}(\ell-1)^{\# T}.\]
\end{proof}
As \( T \) ranges over finite subsets of \( \mathfrak{T}_{E,L} \setminus Z \), the sets \( W_T \) are mutually disjoint. To see this in detail, let \( T \) and \( T' \) be distinct finite subsets of \( \mathfrak{T}_{E,L} \setminus Z \). Without loss of generality, suppose that \( T \not\subset T' \). Then there exists a prime \( q \in T \setminus T' \). By construction, every element of \( W_T \) is ramified at \( q \), whereas every element of \( W_{T'} \) is unramified at \( q \). It follows that \( W_T \cap W_{T'} = \emptyset \). Let \( W \) denote the disjoint union \( W = \bigsqcup_T W_T \) as $T$ ranges over finite subsets of $\mathfrak{T}_{E,L}\backslash Z$.

\begin{proof}[Proof of Theorem \ref{main thm of section 4}]
 Let $\widetilde{L}/L/\Q$ be defined as in the start of Section \ref{section 4.2} and $Z$ be the finite set of primes defined above (see Definition \ref{defn of Z}). Set
 \[g(s):=\sum_{T} (\ell-1)^{\# T}\left(\prod_{q\in T} q\right)^{-s}=\prod_{q\in \mathfrak{T}_{E,L}\backslash Z} \left(1+(\ell-1)q^{-s}\right),\]
 where $T$ ranges over finite subsets of $\mathfrak{T}_{E,L}\backslash Z$. Write $g(s)=\sum_{n\geq 1} a_n n^{-s}$ where $a_n:=(\ell-1)^r$ if $n$ is a product of $r$ distinct primes in $\mathfrak{T}_{E, L}\backslash Z$, and $a_n:=0$ otherwise. 
 
 \par Arguing as in \cite[Theorem 2.4, p.5]{serredivisibilite}, we find that
\[\log g(s)=(\ell-1)\sum_{q\in \mathfrak{T}_{E,L}} \frac{1}{q^{-s}}+\theta_1(s)\] where $\theta_1$ is holomorphic on $\op{Re}s \geq 1$. Moreover, \[\log g(s) = (\ell-1) \alpha \, \log\left(\frac{1}{s-1}\right)+\theta_2(s), \]
where $\alpha:=\mathfrak{d}(\mathfrak{T}_{E,L})$. It follows from Lemma \ref{density lemma} that 
\[(\ell-1)\alpha=\delta=\frac{\ell^2 - \ell - 1}{\ell^{n-1}(\ell^2-1)}.\]
Thus, we find that 
\[g(s)=(s-1)^{- \delta} \, h(s),\] where $h(s)$ is a non-zero holomorphic function in $\op{Re}(s)\geq 1$. By Delange's Tauberian theorem (cf. \cite[Theorem 7.28]{Tenenbaum}), 
\begin{equation}\label{main thm eqn 1}\sum_{n\leq X} a_n \gg  X(\log X)^{\delta-1}.\end{equation}
By Lemma \ref{L twist f lemma}, $\op{Sel}_\ell(E/\widetilde{L}^f)=0$ for all $f\in W$ and thus,\begin{equation}\label{main thm eqn 2}
     \mathcal{M}(G, E, X)\geq \#\{f\in W\mid |\Delta_{\widetilde{L}^f}|\leq X\}.
 \end{equation}
 \noindent According to Lemma \ref{L twist f lemma}, we have that $ |\Delta_{\widetilde{L}^f}|\leq c_2\left(\prod_{q\in S_0\cup T} q\right)^{\ell^{n-1}(\ell-1)}$, for an absolute constant $c_2>0$. Thus, from \eqref{main thm eqn 2}, we deduce that
 \begin{equation}\label{main thm eqn 3}
     \mathcal{M}(G, E, X)\geq \#\left\{f\in W: \prod_{q\in T} q\leq c_3 X^{\frac{1}{\ell^{n-1}(\ell-1)}}\right\},
 \end{equation}
 where $c_3>0$ is an explicit constant given by 
\[c_3:=\left(c_2^{\frac{1}{\ell^{n-1}(\ell-1)}} \prod_{q\in S_0} q\right)^{-1}.\]

Note that 
\begin{equation}\#\left\{f\in W: \prod_{q\in T} q\leq c_3 X^{\frac{1}{\ell^{n-1}(\ell-1)}}\right\}= \sum_{n\leq c_3 X^{\frac{1}{\ell^{n-1}(\ell-1)}}} a_n\end{equation}
and thus from \eqref{main thm eqn 1} and \eqref{main thm eqn 3}, we deduce that 
 \[\mathcal{M}(G, E, X)\gg X^{\frac{1}{(\ell-1)\ell^{n-1}}}(\log X)^{\delta-1}.\]
 \end{proof}
    
\bibliographystyle{alpha}
\bibliography{references}
\end{document}